\documentclass[11pt,a4paper]{amsart}

\usepackage[T1]{fontenc}
\usepackage{mathtools}
\usepackage{amsmath,amscd}
\usepackage{amssymb}
\usepackage{amsthm}

\usepackage{comment}
\usepackage{enumitem}
\usepackage{graphicx}

\usepackage{mathrsfs}
\usepackage[ocgcolorlinks, linkcolor=blue]{hyperref}

\usepackage[
backend=biber,
style=alphabetic,doi=false,isbn=false,url=false,minalphanames=5,maxalphanames=10,giveninits=true,maxbibnames=99
]{biblatex}
\usepackage{calc}
{\begin{list}{\arabic{enumi}.}{\usecounter{enumi}%
			\setlength{\labelsep}{0.5em}%
			\settowidth{\labelwidth}{(\arabic{enumi})}%
	\setlength{\leftmargin}{\labelwidth+\labelsep}}}%
{\end{list}}

\usepackage{comment}

\DeclareRobustCommand{\SkipTocEntry}[5]{}

\usepackage{xcolor}
\definecolor{LOcolor}{RGB}{150,100,0}

\newtheorem{Theorem}{Theorem}[section]

\newtheorem{Lemma}[Theorem]{Lemma}

\newtheorem{Proposition}[Theorem]{Proposition}

\theoremstyle{definition}

\newtheorem{Remark}[Theorem]{Remark}

\numberwithin{equation}{section}

\newcommand{\mR}{\mathbb{R}}                    
\DeclarePairedDelimiterX{\seminorm}[1]{[}{]}{#1}
\DeclarePairedDelimiterX{\norm}[1]{\lVert}{\rVert}{#1}
\DeclarePairedDelimiterX{\inner}[2]{\langle}{\rangle}{\,#1,\,#2}
\DeclarePairedDelimiterX{\abs}[1]{\lvert}{\rvert}{#1}

\newcommand{\ol}[1]{\overline{#1}}

\newcommand{\eps}{\varepsilon}

\newcommand{\p}{\partial}

\newcommand{\half}{\frac{1}{2}}

\newcommand{\e}{\varepsilon}
\newcommand{\Om}{\Omega}

\newcounter{sidenote}

\begin{document}

\title{Inverse problems for semilinear elliptic PDE with a general nonlinearity $a(x,u)$} 

\author[D. Johansson]{David Johansson}
\address{Department of Mathematics, Aarhus University}
\email{johansson@math.au.dk}

\author[J. Nurminen]{Janne Nurminen}
\address{Computational Engineering, School of Engineering Sciences, Lappeenranta-Lahti University of Technology, Finland \& Department of Mathematics and Statistics, University of Jyväskylä, Jyväskylä, Finland}
\email{janne.s.nurminen@jyu.fi, janne.nurminen@lut.fi}

\author[M. Salo]{Mikko Salo}
\address{Department of Mathematics and Statistics, University of Jyäskylä}
\email{mikko.j.salo@jyu.fi}




\begin{abstract}
This article studies the inverse problem of recovering a nonlinearity in an elliptic equation $\Delta u + a(x,u) = 0$ from boundary measurements of solutions. Previous results based on first order linearization achieve this under a sign condition on $\p_u a(x,u)$, and results based on higher order linearization recover the Taylor series of $a(x,u)$ with respect to $u$. We improve these results and show that a general nonlinearity, and not just its Taylor series, is uniquely determined up to gauge near a fixed solution. Our method is based on constructing a good solution map that locally parametrizes solutions of the nonlinear equation by solutions of the linearized equation.

\medskip
		
\end{abstract}

\maketitle

\section{Introduction}

\subsection*{Motivation}

Let $\Omega \subseteq \mR^n$, $n\geq2$, be a bounded domain whose boundary is assumed to be $C^{\infty}$ for simplicity, and let $a \in C^{k}(\mR, C^{2,\alpha}(\ol{\Omega}))$ where $k \geq 3$ and $0 < \alpha < 1$. We write $a = a(x,z)$ for $x \in \ol{\Om}$ and $z \in \mR$, and consider equations of the form 
\begin{equation} \label{semilinear_eq}
	\Delta u(x) + a(x, u(x)) = 0 \text{ in }\Omega.
\end{equation}
In this article we study the inverse problem of identifying the function \( a(x,z) \) from certain boundary measurements of solutions of \eqref{semilinear_eq}. For example, the boundary measurements could be encoded by a Dirichlet-to-Neumann (DN) map if the equation is well-posed, or more generally one could use the (full) Cauchy data set 
\[
	C_a := \{ (u|_{\p \Omega}, \p_{\nu} u|_{\p \Omega}) \,:\, u \in C^{2,\alpha}(\ol{\Omega}) \text{ solves $\Delta u + a(x,u) = 0$} \}.
\]
That is, we wish to answer the question:
\begin{center}
	Does \( C_{a} \) determine \( a(x,z) \)?
\end{center}

For linear equations $\Delta u + q(x) u = 0$, the question above is a version of Calderón's inverse problem and there is large literature (see e.g.\ the survey \cite{Uhlmann_survey}). There are also many results for nonlinear equations. The first generation of such results was based on \emph{first order linearization}, i.e.\ on studying the (first) Fréchet derivative of the nonlinear DN map and using existing results for linear equations. This method was introduced in \cite{Isakov1993}, and further results for determining a nonlinearity $a(x,u)$ as in \eqref{semilinear_eq} were given in \cites{IS, IN, ImanuvilovYamamoto2013}. These results typically require assumptions such as 
\begin{gather} 
a(x,0)= 0, \label{a_zero} \\
\p_u a(x,u) \leq 0, \label{a_positivity}
\end{gather}
which ensure well-posedness and a maximum principle. The assumption \eqref{a_positivity} was weakened in \cite{IN}, and  \cite{Sun} gave a result without assuming \eqref{a_zero}. The results show that one can recover $a(x,u)$ in (some subset of) the \emph{reachable set} 
\[
E_a := \{ (x,z) \,:\, x \in \ol{\Om}, \ z = u(x) \text{ for some solution $u$ of $\Delta u + a(x,u) =0$} \}.
\]
There are many related works for quasilinear and conductivity type equations. References may be found in the survey articles \cites{Sun2005, Uhlmann2009}.

The works \cites{FO, LLLS} introduced a higher order linearization method in inverse problems for nonlinear elliptic equations, motivated by the earlier work \cite{KLU} for hyperbolic equations. This method applies to inverse problems for equations like \eqref{semilinear_eq} without any positivity assumptions as in \eqref{a_positivity}. Moreover, unlike in the first order linearization method that reduced matters to known results for linear equations, in higher order linearization the nonlinearity is used as a tool that helps in solving inverse problems. In this way, one can obtain results in partial data problems \cites{LLLS2, KU, ST} or anisotropic problems \cites{FO, LLLS, CFO23, FKO23} that are stronger than the known results for corresponding linear equations.

However, the higher order linearization results for \eqref{semilinear_eq} start with the assumptions that  
\begin{gather}
a(x,0) = 0, \label{a_zero_h} \\
\text{$0$ is not a Dirichlet eigenvalue for $\Delta + \p_u a(x,0)$ in $\Omega$.} \label{a_wellposed_h}
\end{gather}
The first assumption ensures that $u \equiv 0$ is a solution of \eqref{semilinear_eq}. The second assumption ensures that the linearized equation is well-posed for small Dirichlet data, and hence there is a nonlinear DN map $\Lambda_a$ for \eqref{semilinear_eq} defined for small Dirichlet data. The additional assumption $\p_u a(x,0) = 0$ also appears in many results. The works \cites{FO, LLLS} then show that $\Lambda_a$ (defined for small Dirichlet data) determines $\p_u^j a(\,\cdot\,,0)$ for many $j \geq 0$. If one additionally assumes that $a(x,z)$ is real-analytic in $z$, then this is sufficient for determining $a(x,z)$ completely.

\subsection*{Results}

Our aim is to consider inverse problems for \eqref{semilinear_eq} for general functions $a \in C^{k}(\mR, C^{2,\alpha}(\ol{\Omega}))$. In particular, we wish to remove the assumptions \eqref{a_zero}--\eqref{a_positivity} in the first order linearization method and \eqref{a_zero_h}--\eqref{a_wellposed_h} in the higher order linearization method. This requires certain changes in the problem setup. First of all, the results in \cites{FO, LLLS} are based on looking at solutions of \eqref{semilinear_eq} close to $u \equiv 0$ and on well-posedness for small data. Moreover, the linearized equation might not be well-posed in general, but by Fredholm theory it is still well-posed for most Dirichlet data (i.e.\ data that are $L^2$-orthogonal to a finite dimensional space). It follows that there may not be a Dirichlet-to-Neumann map to work with.

For these reasons, in the general case we consider an arbitrary but fixed function $w \in C^{2,\alpha}(\ol{\Om})$ and for $\delta > 0$ we define the local Cauchy data set 
\begin{equation*}
    C_{a}^{w,\delta} := \{ (u|_{\p \Omega}, \p_{\nu} u|_{\p \Omega}) \,:\, u \in C^{2,\alpha}(\ol{\Omega}) \text{ solves $\Delta u + a(x,u) = 0$}\, \text{ and }\, \norm{w-u}_{C^{2,\alpha}(\ol{\Om})}\leq \delta\}.
\end{equation*}
If $w \equiv 0$, this would be analogous to small Dirichlet data. By the Fredholm theory fact mentioned above one expects that there are many solutions close to $w$, if \( w \) itself is a solution, and we will prove a precise version of such a result.

Our first main theorem shows that if two nonlinearities $a_1$ and $a_2$ admit a common solution $w$ and if their local boundary measurements satisfy the inclusion $C_{a_1}^{w,\delta} \subseteq C_{a_2}^{0,C}$, then $a_1 = a_2$ near the common solution $w$.

\begin{Theorem} \label{thm_main0}
Let $a_1, a_2 \in C^3(\mR, C^{2,\alpha}(\ol{\Om}))$, and let $w \in C^{2,\alpha}(\ol{\Om})$ solve $\Delta w + a_1(x,w) = 0$ and $\Delta w + a_2(x,w) = 0$ in $\Om$. If for some $\delta, C > 0$ one has 
\[
C_{a_1}^{w,\delta} \subseteq C_{a_2}^{0,C},
\]
then there is $\eps > 0$  such that 
\[
a_1(x,w(x)+\lambda) = a_2(x,w(x)+\lambda), \qquad x \in \ol{\Om}, \ \ \abs{\lambda} \leq \eps.
\]
\end{Theorem}

This theorem is based on using first order linearization and it is valid for general nonlinearities. In particular, the sign condition \eqref{a_positivity} in the earlier results mentioned above is not needed. We note that the higher order linearization method does not require the sign condition either, but the result in \cites{FO, LLLS} for the case $w \equiv 0$ was $\p_u^j a_1(\,\cdot\,,0) = \p_u^j a_2(\,\cdot\,,0)$ for many $j \geq 0$ which is clearly weaker than the conclusion in Theorem \ref{thm_main0}. Moreover, if the linearizations of $\Delta u + a_j(x,u) = 0$ (linearized at the solution $w$) happen to be well-posed, then by the arguments in Section \ref{sec_solvability} there are Dirichlet-to-Neumann maps $\Lambda_{a_j}$ defined for Dirichlet data close to $w|_{\p \Om}$, and then by Theorem \ref{thm_main0} one obtains $a_1=a_2$ near points $(x,w(x))$ whenever $\Lambda_{a_1} = \Lambda_{a_2}$.

\begin{Remark}
It is in order to explain the assumption $C_{a_1}^{w,\delta} \subseteq C_{a_2}^{0,C}$ in the theorem. A more typical way of stating a uniqueness result would be to say that $C_{a_1} = C_{a_2}$ (i.e.\ the full Cauchy data sets of $a_1$ and $a_2$ agree) implies $a_1 = a_2$ somewhere. However, $C_{a_1} = C_{a_2}$ implies our assumption $C_{a_1}^{w,\delta} \subseteq C_{a_2}^{0,C}$ in many cases, e.g.\ when the equation $\Delta u + a_2(x,u) = 0$ is well-posed for Dirichlet data near $w|_{\p \Omega}$ or when $\norm{\p_u a_2}_{L^{\infty}(\Omega \times \mR)} < \infty$ (the latter fact follows from the Cauchy data estimates in Section \ref{sec_quant_ucp}). This assumption is nonetheless enough for our purposes.

More precisely, the assumption $C_{a_1}^{w,\delta} \subseteq C_{a_2}^{0,C}$  means that if $u_1$ solves $\Delta u_1 + a_1(x,u_1) = 0$ with $\norm{u_1-w}_{C^{2,\alpha}(\ol{\Om})} \leq \delta$, then there is $u_2$ solving $\Delta u_2 + a_2(x,u_2) = 0$, having the same Cauchy data as $u_1$, and satisfying $\norm{u_2}_{C^{2,\alpha}(\ol{\Om})} \leq C$. The existence of such a constant $C$ is required in the proof to make sure that if $u_1$ is very close to $w$, then $u_2$ will also be close to $w$ and we can use a uniqueness result to guarantee that $u_2$ can be differentiated with respect to some parameters if the same is true for $u_1$.

We also note that when \( a_{k}(x,z) \) are linear in \( z \) and \( C^{2,\alpha} \) in \( x \), for \( k=1,2 \), then the condition \( C_{a_{1}} =  C_{a_{2}} \) is necessary and sufficient for \( C_{a_{1}}^{w,\delta}\subseteq C_{a_{2}}^{0,C} \). Sufficiency holds by Cauchy data estimates, as noted above. For the necessity, Theorem \ref{thm_main0} implies that \( a_{1} = a_{2} \) in an open subset of \( \Omega\times \mathbb{R} \). The linearity in \( z \) implies that \( a_{1} = a_{2} \) in all of \( \Omega\times\mathbb{R} \). Hence, \( C_{a_{1}} = C_{a_{2}} \).

\end{Remark}

If the nonlinearities $a_1$ and $a_2$ do not admit a common solution, then this inverse problem has a gauge invariance as observed in \cite{Sun}. If $a \in C^{k}(\mR, C^{\alpha}(\ol{\Om}))$ is a nonlinearity and $\varphi \in C^{2,\alpha}(\ol{\Om})$ is any function satisfying $\varphi|_{\p \Om} = \p_{\nu} \varphi|_{\p \Om} = 0$, we define 
\begin{equation}\label{transformation_T}
T_{\varphi} a(x,u) := \Delta \varphi(x) + a(x,u + \varphi(x)).
\end{equation}
Then $u$ solves $\Delta u + a(x,u(x)) = 0$ if and only if $v = u-\varphi$ solves $\Delta v + T_{\varphi} a(x, v(x)) = 0$. It follows that the solutions of these two equations have the same Cauchy data. Hence, if the Cauchy data sets for $a_1$ and $a_2$ agree, one can only expect that $a_2(x,u) = T_{\varphi} a_1(x,u)$ for $(x,u)$ in the reachable set. There are a number of related works based on the first linearization, see e.g.\ \cites{IS, IN, Sun}.
 Recent works that involve a similar gauge invariance are given in \cites{LL23, KLL23}. We also mention the examples in \cites{IS, FKO23} showing that in general the reachable set is not all of $\ol{\Om} \times \mR$.

Our next result shows that if one knows the Cauchy data for a nonlinearity $a$ and for solutions close to a given solution $w$, then one can recover $a$ near points $(x,w(x))$ precisely up to the gauge mentioned above.

\begin{Theorem}\label{thm_main01}
Let $a_1, a_2 \in C^{3}(\mR, C^{2,\alpha}(\ol{\Om}))$, and let $w_1 \in C^{2,\alpha}(\ol{\Om})$ solve $\Delta w_1 + a_1(x,w_1) = 0$ in $\Om$. If 
\[
C_{a_1}^{w_1,\delta} \subseteq C_{a_2}^{0,C}
\]
for some $\delta, C > 0$, then there is $\eps > 0$  such that 
\begin{gather*}
a_1(x,w_1(x)+\lambda) = T_{\varphi} a_2(x,w_1(x)+\lambda)
\end{gather*}
whenever $x \in \ol{\Om}$ and $\abs{\lambda} \leq \eps$. Moreover, under our assumptions there is a unique solution $w_2 \in C^{2,\alpha}(\ol{\Om})$ of $\Delta w_2 + a_2(x,w_2) = 0$ in $\Om$ with $w_1|_{\p \Om} = w_2|_{\p \Om}$ and $\p_{\nu} w_1|_{\p \Om} = \p_{\nu} w_2|_{\p \Om}$, and the function $\varphi$ is given by $\varphi = w_2 - w_1$.
\end{Theorem}

Again, Theorem \ref{thm_main01} is valid for general nonlinearities. Note that Theorem \ref{thm_main0} is a corollary of Theorem \ref{thm_main01} since $w_2=w_1$ in that case. In Theorems \ref{thm_main0} and \ref{thm_main01} the exact dependence of $\e$ on different factors is difficult to keep track of. This is mainly because of a compactness argument that we use.

Both Theorem \ref{thm_main0} and \ref{thm_main01} are based on first order linearization and they rely on the solution of an inverse problem for the linearized equation. In contrast, many of the results based on higher order linearization do not rely directly on the inverse problem for the linearized equation. In fact in these results the equation often has a form where the unknown quantities only appear in higher linearizations and not in the first linearization. For such equations, nonlinearity often helps and one can obtain improved results in the presence of nonlinearity.

In the case of the higher order linearization method, we can remove the assumptions \eqref{a_zero_h}--\eqref{a_wellposed_h} that were present in most of the earlier results. The following result is an example of what one can prove.

\begin{Theorem}\label{thm_main02}
Let $a_1, a_2 \in C^{k+1}(\mR, C^{2,\alpha}(\ol{\Om}))$ with $k \geq 2$, let $w_1 \in C^{2,\alpha}(\ol{\Om})$ solve $\Delta w_1 + a_1(x,w_1) = 0$ in $\Om$, and suppose that 
\[
C_{a_1}^{w_1,\delta} \subseteq C_{a_2}^{0,C}
\]
for some $\delta, C > 0$. Let $w_2 \in C^{2,\alpha}(\ol{\Om})$ be the unique solution of $\Delta w_2 + a_2(x,w_2) = 0$ in $\Om$ with $w_1|_{\p \Om} = w_2|_{\p \Om}$ and $\p_{\nu} w_1|_{\p \Om} = \p_{\nu} w_2|_{\p \Om}$. Assume further that 
\begin{equation} \label{thm02_assumption}
\p_u^l a_1(x,w_1) = \p_u^l a_2(x,w_2), \qquad 1 \leq l \leq k-1.
\end{equation}
Then 
\begin{gather*}
\int_{\Om} (\p_u^k a_1(x,w_1) - \p_u^k a_2(x,w_2)) v_1 \ldots v_{k+1} \,dx = 0
\end{gather*}
for any $v_j$ solving the linear equation $\Delta v_j + \p_u a_1(x,w_1) v_j = 0$ in $\Om$.
\end{Theorem}

In other words, if $C_{a_1}^{w_1,\delta} \subseteq C_{a_2}^{0,C}$ and \eqref{thm02_assumption} holds, then $\p_u^k a_1(x,w_1) - \p_u^k a_2(x,w_2)$ is $L^2$-orthogonal to products of $k+1$ solutions of the same linear equation. This is a typical conclusion in the higher order linearization method. Under our current assumptions, Theorem \ref{thm_main01}, which is based on solving the inverse problem for the linearized equation, already implies that $\p_u^k a_1(x,w_1) = \p_u^k a_2(x,w_2)$. The point is that as long as \eqref{thm02_assumption} holds (this is true e.g. for polynomial nonlinearities $a_j(x,u) = q_j(x) u^k$ and $w_j \equiv 0$), Theorem \ref{thm_main02} does not rely on solving the inverse problem for the linearized equation and one can prove this without assuming \eqref{a_zero_h}--\eqref{a_wellposed_h}.
Theorem \ref{thm_main02} remains valid for more general equations such as $\Delta_g u + a(x,u) = 0$ with a smooth Riemannian metric $g$, for which the linearized case is not fully understood. For these more general equations one might be able to obtain improved results in the nonlinear case as was done in \cites{FO, LLLS} and subsequent works.

\subsection*{Methods}

Let us next describe the methods for proving the above results. The first objective is to show that near a solution $w$ of $\Delta w + a(x,w) = 0$, there are many solutions $u_v = w + v + O(\norm{v}^2)$ of the same equation that are parametrized by small solutions $v$ of the linearized equation 
\[
\Delta v + \p_u a(x,w) v = 0.
\]
It will also be important that $u_v$ depends smoothly on $v$. We will prove this by a standard argument using the implicit function theorem. This is slightly delicate since the linearized equation may not be well-posed. In order to have a solution $u_v$ depending smoothly on $v$ that is unique in a suitable sense, one needs to use a solution operator for the linearized equation that takes into account the finite dimensional obstructions to solvability coming from Fredholm theory.

One can recast the previous result in a different language. If $q(x) = \p_u a(x,w(x))$ and 
\begin{align*}
V_q &= \{ v \in C^{2,\alpha}(\ol{\Om}) \,:\, \Delta v + q v = 0 \text{ in $\Om$} \},\\
W_a &= \{ u \in C^{2,\alpha}(\ol{\Om}) \,:\, \Delta u + a(x,u) = 0 \text{ in $\Om$} \},
\end{align*}
then $V_q$ is the solution space of $\Delta + q$ and $W_a$ is a Banach manifold in $C^{2,\alpha}(\ol{\Om})$ consisting of solutions of the nonlinear equation. Then $V_q$ is the tangent space of $W_a$ at $w$, and our result shows that 
\[
S_{a,w}: v \mapsto u_v
\]
is a bijective smooth map from a neighborhood of $0$ in $V_q$ onto a neighborhood of $w$ in $W_a$ with $D S_{a,w}(0) = \mathrm{Id}$. Similar ideas appear e.g.\ in \cites{Eells, Palais, Sunada}.

Next, starting from the assumption $C_{a_1}^{w,\delta} \subseteq C_{a_2}^{0,C}$, we construct a solution $u_{1,v}$ as above for $a_1$, and use the inclusion of Cauchy data sets to conclude that there is a solution $u_{2,v}$ for $a_2$ having the same Cauchy data as $u_{1,v}$. We know that $u_{1,v}$ depends smoothly on $v$, but this is not known for $u_{2,v}$. In order to show that also $u_{2,v}$ depends smoothly on $v$, we prove quantitative estimates showing that solutions of both the linearized and nonlinear equations depend continuously on their Cauchy data. For one of these results we invoke a standard Carleman estimate. Since $u_{2,v}$ depends continuously on its Cauchy data, and since the Cauchy data of $u_{2,v}$ is the same as for $u_{1,v}$ and the latter depends smoothly on $v$, we are able to show that also $u_{2,v}$ depends smoothly on $v$. This also uses certain functional analytic arguments following \cite{OSSU}.

The first order linearization result, Theorem \ref{thm_main0}, is proved as follows. If $C_{a_1}^{w,\delta} \subseteq C_{a_2}^{0,C}$ and if $u_{1,v_t}$ and $u_{2,v_t}$ are as described above with $v_t = v+th$, we differentiate the equations $\Delta u_{j,v_t} + a_j(x,u_{j,v_t})=0$ with respect to $t$, subtract the resulting equations, and integrate against a solution of $\Delta \tilde{v}_2 + \p_u a_2(x,u_{2,v}) \tilde{v}_2 = 0$ to obtain 
\[
\int_{\Om} (\p_u a_1(x,u_{1,v}) - \p_u a_2(x,u_{2,v})) \tilde{v}_1 \tilde{v}_2 \,dx = 0
\]
where $\tilde{v}_1 = DS_{a_1,w}(v)h$. Then we use the bijectivity of $S_{a,w}$ above to conclude that any solution $\tilde{v}_1$ of $\Delta \tilde{v}_1 + \p_u a_1(x,u_{1,v}) \tilde{v}_1 = 0$ can be written as $DS_{a_1,w}(v)h$ for some $h$. Density of products of solutions as in the standard Calder\'on problem \cites{SylvesterUhlmann, Bukhgeim, btw20} implies that 
\[
\p_u a_1(x,u_{1,v}) = \p_u a_2(x,u_{2,v}) \text{ for any small $v \in V_q$.}
\]
Next we show that $\varphi_v \coloneqq u_{2,v}-u_{1,v}$ is independent of $v$, by observing that the derivative of $\varphi_v$ with respect to $v$ is identically $0$, because it solves a linear elliptic equation and has zero Cauchy data. Since $\varphi_0 = w-w = 0$, we obtain 
\[
\p_u a_1(x,u_{1,v}) = \p_u a_2(x,u_{1,v}) \text{ for any small $v \in V_q$.}
\]
(In the setting of Theorem \ref{thm_main01} one has $u_{2,v} = u_{1,v} + \varphi$ instead.) It then remains to show that for any fixed $x_0$, the values $u_{1,v}(x_0)$ generate an interval $[w(x_0)-\eps, w(x_0)+\eps]$ by varying $v$. This follows since $u_{1,v} = w + v + O(\norm{v}^2)$ and since one can generate linear solutions $v$ with $v(x_0) \neq 0$ by using Runge approximation. This concludes the outline of proof of Theorem \ref{thm_main0}. The proof of Theorem \ref{thm_main01} is analogous, except that $\varphi$ will be a nonzero function that is independent of $v$.

Finally, we use the higher order linearization method and differentiate $k$ times the equations $\Delta u_{j,v} + a_j(x, u_{j,v}) = 0$ with respect to $v$. Subtracting the resulting equations, using the assumption \eqref{thm02_assumption} and integrating against a solution $v_{k+1}$, we arrive at Theorem \ref{thm_main02}.

The article is organized as follows. In Section \ref{sec_solvability} we construct a solution map for equation \eqref{semilinear_eq}. Section \ref{sec_quant_ucp} is dedicated to quantitative uniqueness results for \eqref{semilinear_eq} and its linearization. The linearization methods require two smooth solution maps and the second one is constructed in Section \ref{sec_second_sol_map}. In Section \ref{sec_first_linearization} we use first order linearization to prove Theorems \ref{thm_main0} and \ref{thm_main01}. Finally, in Section \ref{sec_proofs_main} we give the proof of Theorem \ref{thm_main02}. At the end we have an \hyperref[sec_runge]{Appendix} where we give a Runge approximation result for the first linearization of \eqref{semilinear_eq}.

\subsection*{Acknowledgements}

The authors are partly supported by the Academy of Finland (Centre of Excellence in
Inverse Modelling and Imaging, grant 353091). M.S.\ was also supported by the European Research Council under Horizon 2020 (ERC CoG 770924). J.N.\ is also supported by the Research Council of Finland (Flagship of Advanced Mathematics for Sensing Imaging and Modelling grant 359183) and by the Emil Aaltonen Foundation.

\section{A smooth solution operator} \label{sec_solvability}
The linearization methods used in this work are based on constructing solutions $u = u_v$ of 
\begin{equation} \label{eq:nonlinear}
	\Delta u + a(x, u) = 0 \text{ in }\Omega,
\end{equation}
so that $u_v$ is close to a fixed solution $w$ of \eqref{eq:nonlinear} and depends smoothly on a small solution $v$ of the linearized equation
\begin{equation} \label{eq:linearized}
	\Delta v + \partial_{u}a(x, w)v = 0 \text{ in }\Omega.
\end{equation}
In order to parametrize solutions of \eqref{eq:nonlinear} on solutions of \eqref{eq:linearized}, we need to be able to single out suitable solutions of the linearized equation which may not be well-posed.
This is done in Lemma \ref{lemma_linearized_ri} by using the Fredholm alternative.
We then construct solutions $u_v$ of \eqref{eq:nonlinear} by solving a nonlinear fixed point equation.
This fixed point equation is solved in Lemma \ref{lemma_fixed_point}.
The construction of the smooth solution map $v \to u_v$ for the equation \eqref{eq:nonlinear} is completed in Lemma \ref{lemma_perturbed_sol}. Finally in Lemma \ref{lemma_s_derivative} we show that the first Fréchet derivative of the solution operator is an isomorphism between spaces of solutions to the linearized equation.

Before we proceed to the results of this section, let us define some spaces and mappings that are used throughout. Let $\Omega \subseteq \mR^n$ with $n \geq 2$ be a bounded open set with $C^{\infty}$ boundary, and let $q \in C^{1,\alpha}(\ol{\Om})$ be real valued where $0 < \alpha < 1$. First we have the kernel \( N_{q} \) of the operator \( \Delta+q:H_{0}^{1}(\Omega)\to H^{-1}(\Omega) \) and the space of Neumann data \( \partial_{\nu}N_{q} \) of the functions in \( N_{q} \), i.e.\ 
\begin{align*}
N_q &= \{ \psi \in H^1_0(\Omega) \,:\, (\Delta+q)\psi = 0 \}, \\
\p_{\nu} N_q &= \{ \p_{\nu} \psi|_{\p \Omega} \,:\, \psi \in N_q \}.
\end{align*}
These spaces appear due to the use of the Fredholm alternative. These are finite dimensional spaces, and since $q \in C^{1,\alpha}(\ol{\Om})$, elliptic regularity \cite[Theorem 2.2]{browder62} ensures that $N_q \subseteq C^{3,\alpha}(\ol{\Om})$ and $\p_{\nu} N_q \subseteq C^{2,\alpha}(\p \Om)$. The last fact is the only reason why we assume $q \in C^{1,\alpha}(\ol{\Om})$ (otherwise $q \in C^{\alpha}(\ol{\Om})$ would have been sufficient). We let $\{ \p_{\nu} \psi_1, \ldots, \p_{\nu} \psi_m \}$ be an orthonormal basis of $\p_{\nu} N_q$ with respect to the $L^2(\p \Om)$-inner product.

We now show that even in the case when $0$ is a Dirichlet eigenvalue of $\Delta+q$ in $\Om$, the equation $(\Delta+q)u = F$ has a solution $u$ for any $F$ and one can prescribe the Dirichlet data of $u$ in the $L^2(\p \Om)$-orthocomplement of the finite dimensional space $\p_{\nu} N_q$. Below, the notation $\perp$ will always mean $L^2$-orthogonality.

\begin{Lemma}  \label{lemma_linearized_ri}
Let $q \in C^{1,\alpha}(\ol{\Om})$. For any $F \in C^{\alpha}(\ol{\Om})$ and $f \in C^{2,\alpha}(\p \Om)$, there is a unique function $\Phi = \Phi(F,f) \in \p_{\nu} N_q$ such that the problem  
\begin{equation}\label{eq:schrödinger}
\begin{cases}
	\Delta u + qu = F\quad&\text{in }\Omega, \\
	u = f - \Phi \quad&\text{on }\partial\Omega,
\end{cases}
\end{equation}
admits a solution $u \in C^{2,\alpha}(\ol{\Om})$. The function $\Phi$ is given by 
\begin{equation}\label{eq:finite_rank_map}
	\Phi(F, f) = \sum_{j=1}^m \left(\int_{\Om}F\psi_j\,dx + \int_{\partial\Om} f\partial_{\nu}\psi_j\,dS\right)\partial_{\nu}\psi_j.
\end{equation}
Moreover, there is unique solution $u_{F,f} = G_q(F,f)$ such that $u_{F,f} \perp N_q$. The solution $u_{F,f}$ depends linearly on $F$ and $f$ and satisfies 
\begin{equation}\label{uff_estimate}
\norm{u_{F,f}}_{C^{2,\alpha}(\ol{\Om})} \leq C(\norm{F}_{C^{\alpha}(\ol{\Om})} + \norm{f}_{C^{2,\alpha}(\p \Om)}),
\end{equation}
where $C$ is independent of $F$ and $f$.
\end{Lemma}
\begin{proof}
We first consider the case of solvability in $H^2(\Om)$ with zero Dirichlet data. If $X = H^2(\Om) \cap H^1_0(\Om)$ equipped with the $H^2(\Om)$ norm, then by Fredholm theory \cite[Theorem 4 in Section 6.2.3]{Evans} and elliptic regularity \cite[Theorem 4 in Section 6.3.2]{Evans} the map 
\[
T: X \to L^2(\Om), \ Tv = (\Delta+q)v
\]
is Fredholm, i.e.\ it has finite dimensional kernel $N_q$ and its range $\mathrm{Ran}(T) = \{ F \in L^2(\Om) \,:\, F \perp N_q \}$ has finite codimension. It follows that the induced map 
\[
T_1: X/N_q \to \mathrm{Ran}(T)
\]
is bounded and bijective, hence invertible by the open mapping theorem. The space $X/N_q$ can be identified with $Y = \{ u \in X \,:\, u \perp N_q \}$, and $T$ becomes an isomorphism from $Y$ onto $\mathrm{Ran}(T)$. (To see this, let $E: X\to L^2(\Omega)$ be the restriction to $X$ of the $L^2$-orthogonal projection onto $N_q$. Then $X = \operatorname{Ran}(E) \oplus \operatorname{Ker}(E) = N_q \oplus Y$ \cite[Theorem 13.2 b)]{conway}, and the map $Y \to X/N_q$, $u \mapsto [u]$ identifies $Y$ with $X/N_q$.) It follows that for any $F \in \mathrm{Ran}(T)$ there is a unique $v_F \in X$ with $v_F \perp N_q$ such that $T(v_F) = F$. In other words, for any $F \in L^2(\Om)$ with $F \perp N_q$ there is a unique $v_F \in H^2(\Om) \cap H^1_0(\Om)$ with $v_F \perp N_q$ that depends linearly on $F$ and solves 
\[
(\Delta+q)v = F \text{ in $\Om$}, \qquad v|_{\p \Om} = 0,
\]
and one has 
\begin{equation} \label{vg_estimate}
\norm{v_F}_{H^2(\Om)} \leq C \norm{F}_{L^2(\Om)}.
\end{equation}

We can obtain a similar statement in H\"older spaces. Let $F \in C^{\alpha}(\ol{\Om})$ with $F \perp N_q$ and let $v_F$ be as above. By elliptic regularity, $v_F \in C^{2,\alpha}(\ol{\Om})$ and
\begin{equation} \label{vg_estimate2}
\norm{v_F}_{C^{2,\alpha}(\ol{\Om})} \leq C \norm{F}_{C^{\alpha}(\ol{\Om})}.
\end{equation}
(More precisely, from \cite[Theorem 2.2]{browder62} and \cite[Lemma 6.18 and Problem 6.2]{GT} we obtain $v_F \in C^{2,\alpha}(\ol{\Om})$ and
\begin{equation}\label{vg_estimate3}
\norm{v_F}_{C^{2,\alpha}(\ol{\Om})} \leq C(\norm{v_F}_{C(\ol{\Om})} + \norm{F}_{C^{\alpha}(\ol{\Om})}).
\end{equation}
From Theorem $8.15$ and the remark around equation $8.38$ in \cite{GT} it follows that $\norm{v_F}_{C(\ol{\Om})} \leq C\norm{v_F}_{L^2(\Om)}$. Using this with \eqref{vg_estimate} and \eqref{vg_estimate3} yields \eqref{vg_estimate2}.)

We now consider \eqref{eq:schrödinger}. To study the uniqueness of $u$, we fix a bounded extension operator
\[
E_q: C^{2,\alpha}(\p \Om) \to C^{2,\alpha}(\ol{\Om}) \text{ with $E_q h|_{\p \Om} = h$ and $E_q h \perp N_q$ for all $h \in C^{2,\alpha}(\p \Om)$}.
\]
In fact it is enough to take $E_q = (\mathrm{Id}-P_{N_q})E$, where $E$ is any bounded extension operator \cite[Theorem 3.3.3 and eq. 1) on p. 51]{Triebel1983} and $P_{N_q}$ is the $L^2(\Omega)$-orthogonal projection to $N_q$. We see that $u$ solves \eqref{eq:schrödinger} iff $u = E_q(f-\Phi) + v$, where $v$ solves 
\begin{equation}\label{eq:schrödinger3}
\begin{cases}
	\Delta v + qv = \tilde{F} \quad&\text{in }\Omega, \\
	v = 0 \quad&\text{on }\partial\Omega,
\end{cases}
\end{equation}
where we wrote $\tilde{F} = F - (\Delta+q)E_q(f-\Phi)$. We wish to find a function $\Phi \in \p_{\nu} N_q$ such that $\tilde{F} \perp L^2(\Om)$. If $\psi \in N_q$, integrating by parts gives 
\[
\int_{\Om} \tilde{F} \psi \,dx = \int_{\Om} F \psi \,dx + \int_{\p \Om} (f-\Phi) \p_{\nu} \psi \,dS.
\]
Thus $\tilde{F} \perp N_q$ iff $\Phi$ satisfies for all $\psi \in N_q$ the condition 
\[
\int_{\p \Om} \Phi \p_{\nu} \psi \,dS = \int_{\Om} F \psi \,dx + \int_{\p \Om} f \p_{\nu} \psi \,dS.
\]
This holds for $\Phi$ iff $\Phi = \Phi(F,f)$ is given by \eqref{eq:finite_rank_map}. For this $\Phi$, we let $v_{\tilde{F}} \perp N_q$ be the solution of \eqref{eq:schrödinger3} satisfying \eqref{vg_estimate}. Then $u_{F,f} := E_q(f-\Phi(F,f)) + v_{\tilde{F}} \perp N_q$ satisfies the required estimate \eqref{uff_estimate}.
\end{proof}

Next we prove an auxiliary lemma which is used in several places in the remainder of the article.

\begin{Lemma}\label{lemma_a_boundedness}
	Let \( a\in C^{k}(\mR, C^{\alpha}(\ol\Om)) \) and \( f\in C^{\alpha}(\ol{\Omega}) \) and let \( l\leq k-1 \). Then
	\begin{equation}\label{eq_a_boundedness}
		\norm{\partial_{u}^{l}a(x,f(x))}_{C^{\alpha}(\ol{\Omega})} \leq \norm{\partial_{u}^{l}a}_{C([-M,M], C^{\alpha}(\overline{\Omega}))} + C\norm{f}_{C^{\alpha}(\overline{\Omega})}
	\end{equation}
	where \( M = \norm{f}_{C(\ol{\Omega})} \), and $C>0$ depends on both $a$ and $f$.
\end{Lemma}
\begin{proof}
By using the triangle inequality and the fact that \( \mathbb{R}\ni z \mapsto\partial_{z}a(x,z)\in C^{\alpha}(\overline{\Omega}) \) is locally Lipschitz, we estimate
\begin{align*}
	 \norm{\partial_{u}^{l}a(x,f)}_{C^{\alpha}(\ol{\Omega})} = &\sup_{x\in\ol{\Omega}} \abs{\partial_{u}^{l}a(x,f(x))} + \sup_{\substack{x,y\in\ol{\Omega}\\x\neq y}}\frac{\abs{\partial_{u}^{l}a(x,f(x)) - \partial_{u}^{l}a(y,f(y))}}{\abs{x-y}^{\alpha}} \\
	 \leq&\sup_{x\in\ol{\Omega}} \abs{\partial_{u}^{l}a(x,f(x))} + \sup_{\substack{x,y\in\ol{\Omega}\\x\neq y}}\frac{\abs{\partial_{u}^{l}a(x,f(x)) - \partial_{u}^{l}a(x,f(y))}}{\abs{x-y}^{\alpha}} \\
	 &+\sup_{\substack{x,y\in\ol{\Omega}\\x\neq y}}\frac{\abs{\partial_{u}^{l}a(x,f(y)) - \partial_{u}^{l}a(y,f(y))}}{\abs{x-y}^{\alpha}} \\
	 \leq&\sup_{z\in[-M,M]}\sup_{x\in\ol{\Omega}} \abs{\partial_{u}^{l}a(x,z)} + C\sup_{\substack{x,y\in\ol{\Omega}\\x\neq y}}\frac{\abs{f(x) - f(y)}}{\abs{x-y}^{\alpha}} \\
	 &+\sup_{z\in[-M,M]}\sup_{\substack{x,y\in\ol{\Omega}\\x\neq y}}\frac{\abs{\partial_{u}^{l}a(x,z) - \partial_{u}^{l}a(y,z)}}{\abs{x-y}^{\alpha}} \\
	 \leq&\norm{\partial_{u}^{l}a}_{C([-M,M], C^{\alpha}(\overline{\Omega}))} + C\norm{f}_{C^{\alpha}(\overline{\Omega})}.\qedhere
\end{align*}
\end{proof}

Next we study a fixed point equation related to the linearized equation. Below, we let 
\begin{equation} \label{bdelta_def}
B_{\delta} = \{ u \in C^{2,\alpha}(\ol{\Omega}) \,:\, \norm{u}_{C^{2,\alpha}(\ol{\Om})} < \delta \}.
\end{equation}

\begin{Lemma}\label{lemma_fixed_point}
Let $a \in C^{3}(\mR, C^{1,\alpha}(\ol{\Om}))$ and $w\in C^{2,\alpha}(\ol{\Om})$ be a solution of $\Delta w + a(x,w) = 0$ in $\Om.$
Let $q = \p_ua(x,w)$ and let $G_q(F,f)$ be the solution operator of
\begin{equation*}
\begin{cases}
	\Delta u + q u = F& \quad\text{in }\Om \\
	u = f-\Phi(F,f)&\quad\text{on }\p\Om
\end{cases}
\end{equation*}
provided by Lemma \ref{lemma_linearized_ri}.
Define $R_v(r) = R(v+r)$ where $R:C^{2,\alpha}(\ol{\Om})\to C^{\alpha}(\ol{\Om})$ is given by
\begin{equation}\label{def_capital_R}
	R(h)(x) \coloneqq \int_0^1 [\p_u a(x,w(x)+th(x)) - \p_u a(x,w(x))]h(x) \,dt.
\end{equation}
For fixed $v\in C^{2,\alpha}(\ol{\Om})$ define $T_v: C^{2,\alpha}(\ol{\Om})\to C^{2,\alpha}(\ol{\Om})$ by $T_v(r) = -G_q(R_v(r), 0)$. 

Under the above assumptions, there exists $\delta>0$ such that $T_v\vert_{B_{\delta}}: B_{\delta}\to B_{\delta}$ is a contraction. Furthermore, \begin{equation}\label{eq:fixed_point_quad_estimate}
	\norm{T_v(h)}_{C^{2,\alpha}(\ol{\Om})} \leq C\norm{v+h}_{C^{2,\alpha}(\ol{\Om})}^2, \qquad h\in B_{\delta}.
\end{equation}
Consequently there exists a unique $r \in B_{\delta}$ solving the fixed point equation $r = T_v(r)$. The function $r$ is also the unique solution with $r \perp N_q$ of 
\begin{equation}\label{eq:fixed_point_equation}
\begin{cases}
	\Delta r + \p_ua(x,w)r = -R(v+r)& \quad\text{in }\Om \\
	r|_{\p \Om} \in \p_{\nu} N_q &\quad\text{on }\p\Om,
\end{cases}
\end{equation}
and necessarily $r|_{\p \Om} = \Phi(R(v+r),0)$.
\end{Lemma}
\begin{proof}
We first show that $T_v$ maps $B_{\delta}$ into itself when $\delta$ is small enough.
Let $v, r\in B_{\delta}$ where initially $\delta \leq 1$.
By the mapping properties for $G$ in Lemma \ref{lemma_linearized_ri} and the fundamental theorem of calculus, we have 
\begin{align}\label{tvr_estimate1}
	\norm{T_v(r)}_{C^{2,\alpha}(\ol{\Om})}&\leq C \norm{R(v+r)}_{C^{\alpha}(\ol{\Om})} \\\notag
	&= C \norm*{\int_0^1 [\p_u a(x,w+t(v+r)) - \p_u a(x,w)](v+r) \,dt}_{C^{\alpha}(\ol{\Om})}\\\notag
	&= C \norm*{\int_0^1\int_0^1\p_u^2a(x,w+st(v+r))t(v+r)^2\,ds\,dt}_{C^{\alpha}(\ol{\Om})}
\end{align}
From Lemma \ref{lemma_a_boundedness} we get for $s, t \in [0,1]$ that 
\begin{equation}\label{partial_a_estimate}
	\norm{\partial_{u}^{2}a(x,w+st(v+r))}_{C^{\alpha}(\ol{\Omega})} \leq C_{w,a}.
\end{equation}
Since $v\in B_{\delta}$ we have, by using \eqref{partial_a_estimate} in \eqref{tvr_estimate1}, that
\begin{equation*}
	\norm{T_v(r)}_{C^{2,\alpha}(\ol{\Om})}\leq C\norm{v+r}_{C^{\alpha}(\ol{\Om})}^2 \leq  C\norm{v+r}_{C^{2,\alpha}(\ol{\Om})}^2 \leq C\delta^2.
\end{equation*}
The second inequality above proves \eqref{eq:fixed_point_quad_estimate}.
For $\delta$ small enough we get
\begin{equation*}
	\norm{T_v(r)}_{C^{k,\alpha}(\ol{\Om})}\leq \delta
\end{equation*}
and conclude that $T_v$ indeed maps $B_{\delta}$ into itself.

Next we show the contraction property of $T_v.$ Let $r_1,r_2\in B_{\delta}$.
Then, as in \eqref{tvr_estimate1}, we have
\begin{equation*}
	\norm{T_v(r_1) - T_v(r_2)}_{C^{2,\alpha}(\ol{\Om})} \leq C\norm{R_v(r_1) - R_v(r_2)}_{C^{\alpha}(\ol{\Om})}.
\end{equation*}
Denote $u_i=v+r_i$, $i=1,2$.
Then
\begin{align*}
	R_v(r_1)-R_v(r_2) = &\int_0^1\left(\partial_u a(x , w + t u_1) - \partial_u a(x, w)\right) u_1 - \left(\partial_u a(x, w + t u_2)-\partial_u a(x,w)\right) u_2\,dt\\
	= &\int_0^1(\partial_u a(x,w+t u_1)-\partial_u a(x,w)+\partial_u a(x,w+t u_2)-\partial_u a(x,w)) (u_1-u_2) \\
	&-(\partial_u a(x,w+t u_2)-\partial_u a(x,w))u_1+(\partial_u a(x,w+t u_1)-\partial_u a(x,w))u_2\,dt\\
	= &\int_0^1(u_1-u_2)\left(t u_1\int_0^1\partial_u^2a(x,w+st u_1)\,ds+t u_2\int_0^1\partial_u^2a(x,w+st u_2)\,ds\right)\\
	&-t u_1u_2\int_0^1\partial_u^2a(x,w+st u_2)\,ds+t u_1u_2\int_0^1\partial_u^2a(x,w+st u_1)\,ds\,dt\\
	= &\int_0^1(u_1-u_2)\Bigg(t u_1\int_0^1\partial_u^2a(x,w+st u_1)\,ds+t u_2\int_0^1\partial_u^2a(x,w+st u_2)\,ds\\
	&+t^2u_1u_2\int_0^1\int_0^1 s \partial_u^3a(x,w + y s t u_1+(1-y)st u_2)\,dy\,ds\Bigg)\,dt.
\end{align*}
Let us estimate the norm of the last expression term by term.
Since $v,r_i\in B_{\delta}$ then $u_i\in B_{2\delta}$ for $i\in\{1,2\}$.
Using this and \eqref{partial_a_estimate} we have that
\begin{equation*}
\begin{aligned}
	\norm*{t u_i\int_0^1\partial_u^2a(x,w+st u_i)\,ds}_{C^{\alpha}(\ol{\Om})} &\leq t\norm{u_i}_{C^{\alpha}(\ol{\Om})}\int_0^1\norm{\partial_u^2a(x,w+st u_i)}_{C^{\alpha}(\ol{\Om})}\,ds \\
	&\leq t\delta C.
\end{aligned}
\end{equation*}
Just as in \eqref{partial_a_estimate} we get \( \norm{\partial_{u}^{3}a(x,w+h)}_{C^{\alpha}(\ol{\Omega})} \leq C_{w,a} \). Using this we estimate
\begin{equation*}
\begin{aligned}
	&\norm*{t^2u_1u_2\int_0^1\int_0^1 s \partial_u^3a(x,w + y s t u_1+(1-y)st u_2)\,dy\,ds}_{C^{\alpha}(\ol{\Om})} \\
	&\leq t^2\norm{u_1u_2}_{C^{\alpha}(\ol{\Om})}\int_0^1\int_0^1 s \norm{\partial_u^3a(x,w + y s t u_1+(1-y)st u_2)}_{C^{\alpha}(\ol{\Om})}\,dy\,ds \\
	&\leq t^2\delta^2C
\end{aligned}
\end{equation*}
Finally, for small enough $\delta>0$ we have 
\begin{equation*}
\begin{aligned}
	\norm{T_v(r_1) - T_v(r_2)}_{C^{2,\alpha}(\Om)} &\leq C_a \int_0^1\norm{u_1-u_2}_{C^{\alpha}(\ol{\Om})} (2t\delta C + t^2\delta^2C)\,dt\\
	&\leq \half\norm{u_1 - u_2}_{C^{2,\alpha}(\ol{\Om})} \\
	&= \half\norm{r_1 - r_2}_{C^{2,\alpha}(\ol{\Om})}.
\end{aligned}
\end{equation*}
Thus $T_v$ is a contraction and the Banach fixed point theorem ensures existence and uniqueness of solution to the equation $r = T_v(r)$ in $B_{\delta}$. The definition of $T_v$ ensures that $r$ also solves \eqref{eq:fixed_point_equation}.
\end{proof}

We now construct the smooth solution map $S_{a,w}$, which maps small solutions $v$ of the linearized equation $\Delta v + \p_u a(x,w)v = 0$ to solutions $u$ of the nonlinear equation $\Delta u + a(x,u) = 0$ that are close to some fixed solution \( w \). Below, if $F: U \to Y$ is a $C^1$ map where $X$ and $Y$ are Banach spaces and $U \subseteq X$ is open, we will denote its Fréchet derivative by 
\[
DF(x) = F'(x).
\]
Recall that $B_{\delta}$ is given by \eqref{bdelta_def}.

\begin{Lemma} \label{lemma_perturbed_sol}
Let $a \in C^{k}(\mR, C^{1,\alpha}(\ol{\Om}))$, $k \geq 3$.
Let $w \in C^{2,\alpha}(\ol{\Omega})$ be a solution of $\Delta w + a(x,w) = 0$.
Let $q(x) = \p_u a(x,w(x))$.
Then there exist $\delta, C > 0$ and a $C^{k-1}$ map $Q = Q_{a,w}: B_{\delta} \to B_{\delta}$ satisfying 
\begin{align*}
Q(B_{\delta}) &\subseteq N_q^{\perp}, \\
Q(B_{\delta})|_{\p \Omega} &\subseteq \p_{\nu} N_q, \\
Q(0) &= DQ(0) = 0
\end{align*}
and
\begin{equation}\label{lemma_estimate_r}
    \norm{Q(v)}_{C^{2,\alpha}(\ol{\Omega})}\leq C\norm{v}_{C^{2,\alpha}(\ol{\Omega})}^2,
\end{equation}
such that \( S_{a}\colon B_{\delta}\to C^{2,\alpha}(\ol{\Omega}) \) defined by \( u = S_{a,w}(v) = w + v + Q(v) \) is a \( C^{k-1} \) map satisfying 
\begin{equation} \label{u_v_eq}
	\Delta u + a(x,u) = \Delta v + q v \quad\text{in }\Omega
\end{equation}
with \( S_{a,w}'(0)v = v \). In particular, if $v$ solves $\Delta v + qv = 0$, then $u = S_{a,w}(v)$ solves $\Delta u + a(x,u) = 0$.

Conversely, if $\delta$ is small enough, then given any solution $u \in C^{2,\alpha}(\ol{\Omega})$ of $\Delta u + a(x,u) = 0$ with $\norm{u-w}_{C^{2,\alpha}(\ol{\Omega})} \leq \delta$ there exists a unique solution $v \in C^{2,\alpha}(\ol{\Omega})$ of $\Delta v + qv = 0$ such that $u = S_{a,w}(v)$. The function $v$ is explicitly given by 
\begin{equation} \label{u_v_converse}
v = P_{N_q}(u-w) + G_q(0, (u-w)|_{\p \Omega}),
\end{equation}
and one has $\norm{v}_{C^{2,\alpha}(\ol{\Omega})} \leq C \norm{u-w}_{C^{2,\alpha}(\ol{\Omega})}$.
\end{Lemma}

\begin{proof}
We first construct the map $Q$. Let $v \in B_{\delta}$. We look for a solution $u$ of \eqref{u_v_eq} having the form $u = w + v + r$ and formulate a fixed point equation for $r$.
Taylor expansion gives 
\begin{align*}
    \Delta u + a(x,u) &= \Delta(w+v+r) + a(x,w+v+r) \\
		&= \Delta w + \Delta(v+r) + a(x,w) + \p_u a(x,w) (v+r) + R_{v}(r)
\end{align*}
where $R_v(r)(x) = \int_0^1 [\p_u a(x,w(x)+t[v(x)+r(x)]) - \p_u a(x,w(x))][v(x)+r(x)] \,dt$.
Since $w$ is a solution of $\Delta w + a(x,w) = 0$, we see that $u$ solves $\Delta u + a(x,u) = \Delta v + qv$ if $r$ satisfies 
\[
    \Delta r + \p_u a(x,w) r + R_v(r) = 0.
\]
For each $v \in B_{\delta},$ Lemma \ref{lemma_fixed_point} ensures existence and uniqueness of a solution $r = r_v$ in \( B_{\delta} \) with $r \perp N_q$ and $r|_{\p \Om} \in \p_{\nu} N_q$.
Hence the mapping $v\mapsto r_v$ is well-defined for $v\in B_{\delta}$.
Next we use the implicit function theorem to show that this mapping is \( C^{k-1} \).

Let \( F: C^{2,\alpha}(\ol{\Omega})\times C^{2,\alpha}(\ol{\Omega})\to C^{2,\alpha}(\ol{\Omega}) \) be defined by 
\begin{equation*}
    F(v,r)=r-T_v(r)=r+G_q(R_v(r),0).
\end{equation*}
From the definition of \( R_{v} \) and \( G_q \) it follows that $F(0,0)=0$.
Next, \( R_{v} \) is \( C^{k-1} \) since \( \partial_{u}a \in C^{k-1,\alpha}(\mathbb{R} ; C^{1,\alpha}(\ol{\Omega}) ) \).
Consequently, \( F \) is \( C^{k-1} \) since \( G_q \) is linear.
Moreover
\begin{align*}
    &D_rF|_{(0,0)}(h)\\
    &=h+G_q\left(\int_0^1 \partial_u^2a(x,w+t(v+r))(v+r)+\partial_ua(x,w+t(v+r))-\partial_ua(x,w)\,dt|_{(v,r)=(0,0)},0\right)\\
    &=h
\end{align*}
and this is a linear homeomorphism from
\( C^{2,\alpha}(\ol{\Omega}) \)
to itself.
Thus the implicit function theorem \cite[Theorem 4]{hg27} ensures the existence of open balls \( B_{\delta_1},B_{\delta_2} \) and a $C^{k-1}$ map $Q \colon B_{\delta_1}\to B_{\delta_{2}}$ such that
\begin{equation*}
    F(v,Q(v))=0.
\end{equation*}
Since \( r_{v} \) found by Lemma \ref{lemma_fixed_point} above is the unique solution of $F(v,\,\cdot\,)=0$ in \( B_{\delta} \) for \( v \in B_{\delta}\), we conclude that for \( \delta < \min\{\delta_{1},\delta_{2}\} \), \( r_{v} \) belongs to \( B_{\delta_{2}} \) and hence \( Q(v) = r_{v} \).
We have now shown that for each \( v\in B_{\delta} \) there is a unique \( r_{v}\in B_{\delta} \) with $r \perp N_q$ and $r|_{\p \Om} \in \p_{\nu} N_q$ such that \( u_{v} = w+v+r_{v} \) is a solution of the equation $\Delta u_{v} + a(x,u_{v}) = \Delta v + qv$.
Moreover, the map \( v\mapsto u_{v} = S_{a,w}(v) \) is \( C^{k-1} \).

Next we show that $Q$ satisfies the other properties in the statement. 
The estimate \eqref{eq:fixed_point_quad_estimate} implies
\begin{equation*}
    \norm{r}_{C^{2,\alpha}(\ol{\Om})} = \norm{T_v(r)}_{C^{2,\alpha}(\ol{\Om})} \leq C\norm{v}_{C^{2,\alpha}(\ol{\Om})}^2+C\norm{r}_{C^{2,\alpha}(\ol{\Om})}^2
\end{equation*}
and from this we get that
\begin{equation*}
    \norm{v}_{C^{2,\alpha}(\ol{\Om})}^2\geq \norm{r}_{C^{2,\alpha}(\ol{\Om})} - C\norm{r}_{C^{2,\alpha}(\ol{\Om})}^2 \geq \norm{r}_{C^{2,\alpha}(\ol{\Om})}(1-C\delta).
\end{equation*}
For $\delta$ small enough, and since \( Q(v) = r \), we have
\begin{equation*}
	\norm{Q(v)}_{C^{2,\alpha}(\ol{\Om})} \leq C\norm{v}_{C^{2,\alpha}(\ol{\Om})}^2.
\end{equation*}
This proves \eqref{lemma_estimate_r} and shows that $Q(0) = 0$.
Using \eqref{lemma_estimate_r} together with $Q(0) = 0$ implies that $DQ(0) = 0$.
Since $Q(v) = r_v$, we have $Q(v) \in N_q^{\perp}$ and $Q(v)|_{\p \Omega} \in \p_{\nu}N_q$.

We now prove the converse statement. Suppose that $u \in C^{2,\alpha}(\ol{\Om})$ solves $\Delta u + a(x,u) = 0$ in $\Om$ and $\norm{u-w}_{C^{2,\alpha}(\ol{\Om})} \leq \delta$. We write $\tilde{u} = u - w$ and want to construct $v$ solving $\Delta v + qv = 0$ such that $\tilde{u} = v + Q(v)$. Denote by $P_{N_q}$ and $P_{\p_{\nu} N_q}$ the $L^2$-orthogonal projections to the finite dimensional spaces $N_q$ and $\p_{\nu} N_q$, respectively. Motivated by the conditions $Q(B_{\delta}) \subseteq N_q^{\perp}$ and $Q(B_{\delta})|_{\p \Om} \subseteq \p_{\nu} N_q$ we define 
\[
\psi = P_{N_q} \tilde{u},
\]
and let $\varphi$ to be the unique solution given by Lemma \ref{lemma_linearized_ri} of the problem 
\[
\Delta \varphi + q \varphi = 0 \text{ in $\Om$}, \qquad \varphi \perp N_q, \qquad \varphi|_{\p \Om} = (\mathrm{Id}-P_{\p_{\nu} N_q})(\tilde{u}|_{\p \Om}).
\]
Let $v = \varphi + \psi$, which means that $v$ is given by \eqref{u_v_converse}. It follows that $\Delta v + qv = 0$ and $v|_{\p \Om} = (\mathrm{Id}-P_{\p_{\nu} N_q})(\tilde{u}|_{\p \Om})$. We also have 
\[
\norm{v}_{C^{2,\alpha}(\ol{\Om})} \leq C \norm{\tilde{u}}_{C^{2,\alpha}(\ol{\Om})} = C \norm{u-w}_{C^{2,\alpha}(\ol{\Om})}.
\]
It remains to show that $r = \tilde{u} - v = u-w-v$ satisfies $r = Q(v)$. By the above conditions we have $r \perp N_q$ and $r|_{\p \Om} \in \p_{\nu} N_q$, and $r$ satisfies 
\begin{align*}
(\Delta + q) r &= (\Delta+q)(u-w) = -(a(x,u)-a(x,w) - q(u-w)) \\
 &= -(a(x,w+v+r)-a(x,w)-q(v+r)) = -R_v(r).
\end{align*}
The first part of the proof implies that $r = Q(v)$ if $\delta$ is chosen small enough. This proves that $u = S_{a,w}(v)$. To show that $v$ is unique suppose that $u = S_{a,w}(\tilde{v})$ for another solution $\tilde{v}$. Then the definition of $S_{a,w}$ gives  
\[
v-\tilde{v} = Q(\tilde{v}) - Q(v).
\]
Thus $v-\tilde{v} \perp N_q$ and $v-\tilde{v}|_{\p \Om} \in \p_{\nu} N_q$. Since $(\Delta+q)(v-\tilde{v}) = 0$, Lemma \ref{lemma_linearized_ri} implies $v = \tilde{v}$ showing that $v$ is unique.
\end{proof}

\begin{Lemma} \label{lemma_s_derivative}
In the setting of Lemma \ref{lemma_perturbed_sol}, if $v$ is small and solves $\Delta v + qv = 0$ for $q = \p_u a(x,w)$, define 
\begin{align*}
q_v &= \p_u a(x, S_{a,w}(v)), \\
V_{\tilde{q}} &= \{ h \in C^{2,\alpha}(\ol{\Om}) \,:\, \Delta h + \tilde{q} h = 0 \}
\end{align*}
If $v \in V_q$ is small, the map $DS_{a,w}(v)$ is an isomorphism from $V_{q}$ onto $V_{q_v}$.
\end{Lemma}
\begin{proof}
Let $v$ be a small solution of $\Delta v + q v = 0$ and $v_t = v + th$ where $h \in V_{q}$. Then $u_t = S_{a,w}(v_t)$ solves 
\[
\Delta u_t + a(x, u_t) = 0.
\]
Since $u_t$ is $C^1$ in $t$, the function $\dot{u}_0 = \p_t u_t|_{t=0} = DS_{a,w}(v)h$ satisfies 
\[
\Delta \dot{u}_0 + \p_u a(x,S_{a,w}(v)) \dot{u}_0 = 0.
\]
Thus $DS_{a,w}(v)$ maps $V_{q}$ into $V_{q_v}$.

Now suppose that $v \in V_q$ is small and $\tilde{h} \in V_{q_v}$. For $t$ small define $u_t = S_{a,S_{a,w}(v)}(t \tilde{h})$. By the converse part of Lemma \ref{lemma_perturbed_sol}, if $v$ and $t$ are small enough one has $u_t = S_{a,w}(v_t)$ for a unique small solution $v_t \in V_q$, and $v_t$ is given by 
\[
v_t = P_{N_q}(u_t-w) + G_q(0, (u_t-w)|_{\p \Omega}).
\]
In particular, $v_t$ is $C^1$ in $t$, and since $S_{a,w}(v_0) = u_0 = S_{a,S_{a,w}(v)}(0) = S_{a,w}(v)$ uniqueness gives $v_0=v$. Differentiating the identities $u_t = S_{a,w}(v_t)$ and $u_t = S_{a,S_{a,w}(v)}(t \tilde{h})$ and using $DS(0) = \mathrm{Id}$ gives 
\[
DS_{a,w}(v) \dot{v}_0 = \dot{u}_0 = DS_{a,S_{a,w}(v)}(0)\tilde{h} = \tilde{h}.
\]
This shows that $DS_{a,w}(v): V_q \to V_{q_v}$ is surjective.

Finally, suppose that $h \in V_q$ satisfies $DS_{a,w}(v)h=0$. Since $S_{a,w}(v) = w + v + Q_{a,w}(v)$, we have 
\[
h + DQ_{a,w}(v)h = 0.
\]
But $DQ_{a,w}(0)=0$, which implies that $\norm{DQ_{a,w}(v)} \leq 1/2$ when $v$ is sufficiently small. Here we used that $Q$ is a $C^{k-1}$ map where $k \geq 3$.
This implies that $\norm{h} \leq \frac{1}{2} \norm{h}$, showing that $h=0$. Thus $DS_{a,w}(v): V_q \to V_{q_v}$ is bijective and bounded, and by the open mapping theorem it is an isomorphism.
\end{proof}

\section{Estimates for solutions in terms of their Cauchy data} \label{sec_quant_ucp}

In this section we prove estimates for functions in terms of their Cauchy data and in particular for solutions of the nonlinear equation
\begin{equation} \label{eq:nonlinear2}
	\Delta u + a(x, u) = 0 \text{ in }\Omega.
\end{equation}
The estimate for \eqref{eq:nonlinear2} is used in section \ref{sec_second_sol_map} when constructing the second solution map required for the linearization methods. 

First we obtain an auxiliary regularity estimate that is then used to prove the quantitative results.

\begin{Lemma} \label{lemma_cauchy_quant.}
Let $\Omega\subseteq\mR^{n}$ be a bounded open set with $C^{\infty}$ boundary and let $q \in C^{\alpha}(\ol{\Om})$. There is $C > 0$ such that for any $u \in C^{2,\alpha}(\ol{\Om})$ we have 
\[
    \norm{u}_{C^{2,\alpha}(\ol{\Om})} \leq C (\norm{u}_{C^{2,\alpha}(\p \Om)} + \norm{(\Delta+q)u}_{C^{\alpha}(\ol{\Om})} + \norm{u}_{H^1(\Om)}).
\]
\end{Lemma}
\begin{proof}
Consider the Banach space $X = C^{2,\alpha}(\p \Om) \times C^{\alpha}(\ol{\Om}) \times H^1(\Om)$ with norm 
\[
\norm{(f,F,v)}_X = \norm{f}_{C^{2,\alpha}(\p \Om)} + \norm{F}_{C^{\alpha}(\ol{\Om})} + \norm{v}_{H^1(\Om)}.
\]
We define the map 
\[
T: C^{2,\alpha}(\ol{\Om}) \to X, \ \ T(u) = (u|_{\p \Om}, (\Delta+q)u, j(u)),
\]
where $j$ is the inclusion $C^{2,\alpha}(\ol{\Om}) \to H^1(\Om)$. Then $T$ is bounded, linear and injective. We claim that $T$ has closed range. To see this, suppose that $u_j \in C^{2,\alpha}(\ol{\Om})$ and $T(u_j) \to (f, F, v)$ in $X$. Then $u_j \to v$ in $H^1(\Om)$, $u_j|_{\p \Om} \to f$ in $C^{2,\alpha}(\p \Om)$ and $(\Delta+q)u_j \to F$ in $C^{\alpha}(\ol{\Om})$. On the other hand $(\Delta+q)u_j \to (\Delta+q)v$ in $H^{-1}(\Om)$ and $u_j|_{\p \Om} \to v|_{\p \Om}$ in $H^{1/2}(\p \Om)$, and by uniqueness of limits one has $(\Delta+q)v = F$ and $v|_{\p \Om} = f$. By elliptic regularity, the weak solution $v$ satisfies $v \in C^{2,\alpha}(\ol{\Om})$. Thus $(f,F,v) = T(v)$ and $\mathrm{Ran}(T)$ is closed.

We have proved that $T: C^{2,\alpha}(\ol{\Om}) \to \mathrm{Ran}(T)$ is a bounded linear bijection between Banach spaces. By the open mapping theorem it has a bounded inverse $S: \mathrm{Ran}(T) \to C^{2,\alpha}(\ol{\Om})$, and thus for any $u \in C^{2,\alpha}(\ol{\Om})$ one has 
\[
\norm{u}_{C^{2,\alpha}(\ol{\Om})} = \norm{STu}_{C^{2,\alpha}(\ol{\Om})} \leq C \norm{Tu}_{X}.
\]
This proves the claim.
\end{proof}

Next we show a quantitative uniqueness result that follows by combining Lemma \ref{lemma_cauchy_quant.} with the unique continuation principle. This is the used in Section \ref{sec_second_sol_map} related to the first linearization of \eqref{eq:nonlinear2}.

\begin{Lemma} \label{lemma_cauchy_quantitative}
Let $\Omega\subseteq\mR^{n}$ be a bounded open set with $C^{\infty}$ boundary and let $q \in C^{\alpha}(\ol{\Om})$. There is $C > 0$ such that for any $u \in C^{2,\alpha}(\ol{\Om})$ we have 
\[
    \norm{u}_{C^{2,\alpha}(\ol{\Om})} \leq C (\norm{u}_{C^{2,\alpha}(\p \Om)} + \norm{\p_{\nu} u}_{C^{1,\alpha}(\p \Om)} + \norm{\Delta u + qu}_{C^{\alpha}(\ol{\Om})}).
\]
\end{Lemma}
\begin{proof}
We argue by contradiction and assume that for any $m$ there is $u_m$ such that 
\begin{equation} \label{eq_contra_cauchy2}
    \norm{u_m}_{C^{2,\alpha}(\ol{\Om})} > m (\norm{u_m}_{C^{2,\alpha}(\p \Om)} + \norm{\p_{\nu} u_m}_{C^{1,\alpha}(\p \Om)} + \norm{(\Delta+q)u_m}_{C^{\alpha}(\ol{\Om})}).
\end{equation}
On the other hand, Lemma \ref{lemma_cauchy_quant.} implies that 
\[
\norm{u_m}_{C^{2,\alpha}(\ol{\Om})} \leq C (\norm{u_m}_{C^{2,\alpha}(\p \Om)} + \norm{(\Delta+q)u_m}_{C^{\alpha}(\ol{\Om})} + \norm{u_m}_{H^1(\Om)}).
\]
Normalize $u_m$ so that $\norm{u_m}_{H^1(\Om)} = 1$. Then using \eqref{eq_contra_cauchy2} yields 
\begin{align*}
    \norm{u_m}_{C^{2,\alpha}(\ol{\Om})} \leq C(\frac{1}{m}\norm{u_m}_{C^{2,\alpha}(\ol{\Om})} + 1)
\end{align*}
Then $\norm{u_m}_{C^{2,\alpha}(\ol{\Om})} \leq C$ uniformly when $m$ is sufficiently large.

By Theorem $1.34$ in \cite{af03} the embedding $C^{2,\alpha}(\ol{\Om})\to C^{2}(\ol{\Om})$ is compact. Hence there is a subsequence, still denoted $u_m,$ that converges in $C^2(\ol{\Om})$ to some $u\in C^2(\ol{\Om}).$ On the other hand, from \eqref{eq_contra_cauchy2} and the bound $\norm{u_m}_{C^{2,\alpha}(\ol{\Om})} \leq C$ we see that 
\[
u_m|_{\p \Om} \to 0, \qquad \p_{\nu} u_m|_{\p \Om} \to 0, \qquad (\Delta+q)u_m \to 0
\]
in the respective spaces. By uniqueness of limits we have 
 $u|_{\p \Omega} = 0$, $\p_{\nu} u|_{\p \Om} = 0$, and $(\Delta+q)u=0$. Consequently, $u \equiv 0$ by unique continuation, which contradicts $\norm{u}_{H^1(\Om)} = \lim \norm{u_m}_{H^1(\Om)} = 1$.
\end{proof}

Finally, we invoke a Carleman estimate to show that solutions of semilinear equations of the form \eqref{eq:nonlinear2} are uniquely and stably determined by their Cauchy data.

\begin{Lemma} \label{lem_carleman}
Let $a \in C^{2}(\mR, C^{\alpha}(\ol{\Om}))$, and let $u_0 \in C^{2,\alpha}(\ol{\Om})$ solve $\Delta u_0 + a(x,u_0) = 0$ in $\Om$. If  $u \in C^{2,\alpha}(\ol{\Om})$ is any other solution of $\Delta u + a(x,u) = 0$ in $\Om$ and $\norm{u}_{C^{2,\alpha}(\ol{\Om})}$, $\norm{u_0}_{C^{2,\alpha}(\ol{\Om})} \leq M$, then 
\begin{equation}\label{estimate_lem_carleman}
\norm{u-u_0}_{C^{2,\alpha}(\ol{\Om})} \leq C(M,a) (\norm{u-u_0}_{C^{2,\alpha}(\p \Om)} + \norm{\p_{\nu}(u-u_0)}_{C^{1,\alpha}(\p \Om)}).
\end{equation}
\end{Lemma}

\begin{proof}
We use a standard Carleman estimate (see e.g.\ \cite[Theorem 4.1]{ChoulliPDEA2021}): there are $C, \tau_0 > 0$ and $\varphi \in C^{\infty}(\ol{\Om})$ such that when $\tau \geq \tau_0$, one has 
\[
\norm{e^{\tau \varphi} v}_{L^2(\Om)} + \frac{1}{\tau} \norm{e^{\tau \varphi} \nabla v}_{L^2(\Om)}\leq \frac{C}{\tau^{3/2}} \norm{e^{\tau \varphi} \Delta v}_{L^2(\Om)} + C \norm{e^{\tau \varphi} v}_{L^2(\p \Om)}+ \frac{C}{\tau} \norm{e^{\tau \varphi} \nabla v}_{L^2(\p \Om)}
\]
for any $v \in C^2(\ol{\Om})$. We apply this with $v = u-u_0$ and use the fact that 
\begin{equation} \label{deltav_eq}
-\Delta v = a(x,u) - a(x,u_0) = \left[ \int_0^1 \p_u a(x,(1-t)u_0 + t u) \,dt \right] v.
\end{equation}
Since $|u|, |u_0| \leq M$, we get from Lemma \ref{lemma_a_boundedness} that 
\begin{equation}\label{estimate_additional}
\abs{\Delta v(x)} \leq C(M,a) \abs{v(x)}.
\end{equation}
Thus, choosing $\tau = \tau(M,a)$ large but fixed, we get 
\[
\frac{1}{2} \norm{e^{\tau \varphi} v}_{L^2(\Om)} + \frac{1}{\tau} \norm{e^{\tau \varphi} \nabla v}_{L^2(\Om)} \leq C (\norm{e^{\tau \varphi} v}_{L^2(\p \Om)} + \norm{e^{\tau \varphi} \nabla v}_{L^2(\p \Om)}).
\]
Since $c(M,a) \leq e^{\tau \varphi} \leq C(M,a)$, we have  
\begin{equation}\label{L2_size_epsilon}
\norm{v}_{H^1(\Om)} \leq C(M,a) (\norm{v}_{H^1(\p \Om)} + \norm{\p_{\nu}v}_{L^2(\p \Om)}).
\end{equation}
We still need to estimate $\norm{v}_{C^{2,\alpha}(\ol{\Om})}$. First, Lemma \ref{lemma_cauchy_quant.} gives 
\[
\norm{v}_{C^{2,\alpha}(\ol{\Om})} \leq C\left(\norm{v}_{C^{2,\alpha}(\p\Om)} + \norm{\Delta v}_{C^{\alpha}(\ol{\Om})} + \norm{v}_{H^1(\Om)}\right).
\]
From \eqref{deltav_eq} we observe that 
\[
\norm{\Delta v}_{C^{\alpha}(\ol{\Om})} \leq C \left[ \int_0^1 \norm{\p_u a(\,\cdot\,,(1-t)u_0(\,\cdot\,) + t u(\,\cdot\,))}_{C^{\alpha}(\ol{\Om})} \,dt \right] \norm{v}_{C^{\alpha}(\ol{\Om})}.
\]
By using Lemma \ref{lemma_a_boundedness} to estimate the integral from above by a constant depending on \( a,u \) and \( u_{0} \) we have
\[
\norm{\Delta v}_{C^{\alpha}(\ol{\Om})} \leq C \norm{v}_{C^{\alpha}(\ol{\Om})}.
\]
Thus we get 
\begin{equation} \label{v_int_estimate}
    \norm{v}_{C^{2,\alpha}(\ol{\Om})} \leq C\left(\norm{v}_{C^{2,\alpha}(\p\Om)} + \norm{v}_{C^{\alpha}(\ol{\Om})} + \norm{v}_{H^1(\Om)}\right).
\end{equation}
Next, we have by the Sobolev embedding \cite[Theorem 4.12 Part 2]{af03} that $W^{1,s} \subseteq C^{\alpha}$ where $s = \frac{n}{1-\alpha}$. Using this and \cite[Theorem $5.2$ $(3)$]{af03} we obtain that 
\begin{equation*}
    \norm{v}_{C^{\alpha}(\ol{\Om})}\leq C \norm{v}_{W^{1,s}(\Om)}\leq C \norm{v}_{W^{2,s}(\Om)}^{1/2}\norm{v}_{L^s(\Om)}^{1/2}.
\end{equation*}
Then we use interpolation of $L^{p}$-spaces (see for example \cite[Appendix B]{Evans}) to get
\begin{equation*}
    \norm{v}_{L^s(\Om)}\leq C \norm{v}_{L^2(\Om)}^{\lambda}\norm{v}_{L^r(\Om)}^{1-\lambda}
\end{equation*}
for some $r > s$. Estimating the $L^r$- and $W^{2,s}$-norms by the $C^{2,\alpha}$-norm we have 
\begin{equation*}
    \norm{v}_{C^{\alpha}(\ol{\Om})}\leq C \norm{v}_{C^{2,\alpha}(\ol{\Om})}^{(2-\lambda)/2}\norm{v}_{L^2(\Om)}^{\lambda/2}.
\end{equation*}
Using Young's inequality with $\e$ for $p = 2/\lambda$ and $q = p/(p-1)$ gives
\begin{equation*}
    \norm{v}_{C^{\alpha}(\ol{\Om})}\leq C (\varepsilon\norm{v}_{C^{2,\alpha}(\ol{\Om})}^{q(2-\lambda)/2} + C_{\varepsilon}\norm{v}_{L^2(\Om)}) = (\varepsilon\norm{v}_{C^{2,\alpha}(\ol{\Om})} + C_{\varepsilon}\norm{v}_{L^2(\Om)})
\end{equation*}
since $q = \frac{2}{2-\lambda}$.
Using this in \eqref{v_int_estimate} and choosing $\eps > 0$ sufficiently small finally gives 
\[
\norm{v}_{C^{2,\alpha}(\ol{\Om})} \leq C\left(\norm{v}_{C^{2,\alpha}(\p\Om)} + \norm{v}_{H^1(\Om)}\right).
\]
Combining the last estimate with \eqref{L2_size_epsilon} proves the result.
\end{proof}

\section{A smooth solution map with prescribed Cauchy data}\label{sec_second_sol_map}

As mentioned previously, in order to prove the main results, we need to construct two smooth solution maps for the nonlinear equations
\begin{equation} \label{eq:nonlinear3}
	\Delta u + a_{i}(x, u) = 0 \text{ in }\Omega
\end{equation}
for $i=1,2$. In Section \ref{sec_solvability} we constructed the first one. One reason why we cannot use the solution map \( S_{a_{i},w_{i}} \) for both \( i\in\{1,2\} \) is that we need to control the Cauchy data. If \( u_{1} = S_{a_{1},w_{1}}(v_{1}) \) then we would need to find a solution \( u_{2} = S_{a_{2},w_{2}}(v_{2}) \) such that \( u_{1},u_{2} \) have the same Cauchy data. But the solution maps \( S_{a_{i},w_{i}} \) don't provide enough control of the Neumann data to guarantee that this is possible. Another issue, in particular when identifying the first derivatives \( \partial_{u}a_{i}(x,w_{i}) \), is that the method of linearization relies on differentiating both solution maps \( S_{a_{1},w_{1}},S_{a_{2},w_{2}} \) in the same direction \( v \). But in order to use the same parameter \( v \) for both operators, \( v \) needs to solve both linearized equations
\begin{equation}\label{eq:linearized3}
	\Delta v + \partial_{u}a_{i}(x, w_{i})v = 0 \text{ in }\Omega
\end{equation}
for \( i\in\{1,2\} \). However, before having identified the first derivatives, \( \partial_{u}a_{1}(x,w_{1}) = \partial_{u}a_{2}(x,w_{2}) \), we don't know that such functions \( v \) exist. So the goal of this section is to construct a new solution map \( T_{a_{i},w_{i}} \) that resolves these two issues. That is, we aim to construct a smooth solution map \( T_{a_{2},w_{2}} \) for
\begin{equation}
	\Delta u + a_{2}(x, u) = 0 \text{ in }\Omega
\end{equation}
parametrized on solutions \( v \) of
\begin{equation*}
	\Delta v + \partial_{u}a_{1}(x, w_{1})v = 0 \text{ in }\Omega
\end{equation*}
such that \( T_{a_{2},w_{2}}(v) \) and \( S_{a_{1},w_{1}}(v) \) have the same Cauchy data.

Before constructing \( T_{a_{i},w_{i}} \), we establish some preliminary results. The construction of \( T_{a_{i},w_{i}} \) is based on the implicit function theorem. In order to properly define the function to which the implicit function theorem is applied, we require the existence of a certain projection mapping and existence of a bounded inverse of the Schrödinger operator \( \Delta +q \). We first establish these two results and then proceed to construct \( T_{a_{i},w_{i}} \).

\begin{Lemma}\label{lemma:spaces}
Let $q\in C^{\alpha}(\ol{\Om})$.
Then the spaces
\begin{align*}
	Y &= \{ u\in C^{2,\alpha}(\ol{\Omega}) \colon u\vert_{\p\Om} = \p_{\nu}u\vert_{\p\Om} = 0 \} \\
	Z &= \{ (\Delta + q)u \colon u\in Y \}
\end{align*}
are Banach spaces.
\end{Lemma}
\begin{proof}
It follows from the continuity of the mappings $C^{2,\alpha}(\ol{\Om})\ni u \mapsto u\vert_{\p\Om}\in C^{2,\alpha}(\p\Om)$ and $C^{2,\alpha}(\ol{\Om})\ni u \mapsto \p_{\nu}u\vert_{\p\Om}\in C^{1,\alpha}(\p\Om)$ that $Y$ is a Banach space.
To see that $Z$ is a Banach space, let $v_n = \Delta w_n +qw_n\in Z$ be a sequence converging to some $v\in C^{\alpha}(\ol{\Om})$.
Using Lemma \ref{lemma_cauchy_quantitative}, we have
\begin{equation*}
	\norm{w_n - w_m}_{C^{2,\alpha}(\ol{\Om})} \leq C \norm{\Delta w_n+qw_n - \Delta w_m-qw_m}_{C^{\alpha}(\ol{\Om})},
\end{equation*}
so that $w_n$ is a Cauchy sequence in $Y$.
Hence there is some $w\in Y$ with $w_n\to w$ in $C^{2,\alpha}(\ol{\Om})$.
Next,
\begin{equation*}
	\norm{\Delta w_n+qw_n - \Delta w-qw}_{C^{\alpha}(\ol{\Om})} \leq C\norm{w_n - w}_{C^{2,\alpha}(\ol{\Om})}
\end{equation*}
So for $v = \Delta w+qw$, we have $v_n\to v$ in $C^{\alpha}(\ol{\Om})$ and $Z$ is a Banach space.
\end{proof}

The following result shows that there is a bounded projection $P: C^{2,\alpha}(\ol{\Om}) \to Z\cap C^{2,\alpha}(\ol{\Om})$.
If $C^{2,\alpha}(\ol{\Om})$ were a Hilbert space, the existence of a projection would follow from an orthogonal decomposition $C^{2,\alpha}(\ol{\Om}) = Z \oplus W$.
Since $Z$ is the image of $\Delta+q$ acting on functions whose Cauchy data vanishes, the orthocomplement $W$ would be the set of suitable functions $w$ with $(\Delta+q)w=0$.
Thus any $u \in C^{2,\alpha}(\ol{\Om})$ could be written as $u=(\Delta+q)y+w$, where $y \in Y$ and $(\Delta+q)w=0$.
This shows that $y$ needs to satisfy $(\Delta+q)^2 y = (\Delta+q)u$.
This  formal argument turns out to work also in our case.

\begin{Lemma}\label{lemma:projection}
Let $q\in C^{2,\alpha}(\ol{\Om})$ and let $Y$ and $Z$ be as in Lemma \ref{lemma:spaces}.
Then there exists a bounded projection $P\colon C^{2,\alpha}(\ol{\Om})\to Z\cap C^{2,\alpha}(\ol{\Om})$ such that $P(u) = (\Delta + q)y$ where $y \in C^{4,\alpha}(\ol{\Om})$ is the unique solution of 
\begin{equation*}
	\begin{cases}
		(\Delta + q)^2 y = (\Delta+q)u &\text{in }\Omega \\
		y = \p_{\nu}y = 0 &\text{on }\p\Om.
	\end{cases}
\end{equation*}
\end{Lemma}
\begin{proof}
We first show that there is a unique solution $y \in C^{4,\alpha}(\ol{\Om})$.
If $y$ and $\tilde{y}$ are solutions, then $(\Delta+q)^2(y-\tilde{y})=0$, and integrating this equation against $y-\tilde{y}$ gives $(\Delta+q)(y-\tilde{y})=0$.
Since $y-\tilde{y}$ has vanising Cauchy data, we see that $y=\tilde{y}$ and solutions are unique.
Existence of weak solutions $y \in H^2_0(\Omega)$ for the equation $(\Delta+q)^2 y + \gamma y = F \in H^{-2}(\Omega)$, where $\gamma > 0$ is a constant chosen sufficiently large depending on $q$, follows by using the Riesz representation theorem with the coercive bilinear form $B(y,w) = ((\Delta+q)y, (\Delta+q)w)_{L^2(\Omega)} + \gamma(y,w)_{L^2(\Omega)}$ for $y, w \in Y$.
Fredholm theory shows that there is a countable set of eigenvalues where unique solvability could fail, but our uniqueness argument above shows that one has solvability for $(\Delta+q)^2 y = F$.
Elliptic regularity shows that for $u \in C^{2,\alpha}(\ol{\Om})$, one has $y \in C^{4,\alpha}(\ol{\Om})$.

Now that we know that the equation is uniquely solvable, let $u \in C^{2,\alpha}(\ol{\Om})$, and let $y \in C^{4,\alpha}(\ol{\Om})$ be the solution, and define $P(u) = (\Delta+q)y$.
Then $P(P(u)) = P((\Delta + q)y) = (\Delta+q)v$ for the unique solution  $v$ of 
\begin{equation*}
	\begin{cases}
		(\Delta + q)^2 v = (\Delta+q)^2y &\text{in }\Omega \\
		v = \p_{\nu}v = 0 &\text{on }\p\Om.
	\end{cases}
\end{equation*}
Since $y$ has $0$ Cauchy data, $w=v-y$ satisfies
\begin{equation*}
	\begin{cases}
		(\Delta + q)^2 w = 0 &\text{in }\Omega \\
		w = \p_{\nu}w = 0 &\text{on }\p\Om.
	\end{cases}
\end{equation*}
Unique continuation implies that $w=0$ is the unique solution to this equation.
It follows that $v=y$.
Thus $P(P(u))=P(u)$ and $P$ is indeed a projection.
\end{proof}

\begin{Lemma}\label{lemma:schrödinger_bounded_inverse}
Let $q\in C^{\alpha}(\ol{\Om})$ and let $Y$ and $Z$ be as in Lemma \ref{lemma:spaces}.
Then $\Delta + q\colon Y\to Z$ is bounded and bijective and has a bounded inverse $G\colon Z \to Y$.
\end{Lemma}
\begin{proof}
By definition of $Z$, $\Delta+q$ is surjective.
To see injectivity, suppose $u,v\in Y$ and $(\Delta+q)u = (\Delta+q)v$.
Then $w=v-u$ satisfies
\begin{equation*}
\begin{cases}
	(\Delta+q)w = 0&\text{in }\Om \\
	w = 0&\text{on }\p\Om \\
	\p_{\nu}w = 0 &\text{on }\p\Om.
\end{cases}
\end{equation*}
It follows by the unique continuation principle that $w = 0$.
Hence $\Delta+q$ is injective. Lastly, we have
\begin{equation*}
	\norm{(\Delta+q)u}_{C^{\alpha}(\ol{\Om})} \leq \norm{\Delta u}_{C^{\alpha}(\ol{\Om})} + \norm{qu}_{C^{\alpha}(\ol{\Om})} \leq C\norm{u}_{C^{2,\alpha}(\ol{\Om})}
\end{equation*}
so that $\Delta+q$ is bounded.
Now it follows from the open mapping theorem that there exists a bounded inverse $G$ of $\Delta+q$.
\end{proof}

Below we will use the ball $V_{q,\delta}$ in the space of solutions,
\[
V_{q,\delta} = \{ v \in C^{2,\alpha}(\ol{\Om}) \,:\, \Delta v + qv = 0 \text{ and } \norm{v}_{C^{2,\alpha}(\ol{\Om})} < \delta \}.
\]

\begin{Lemma}\label{lemma:smooth_solution_operator}
Let $a_1,a_2\in C^{k+1}(\mR, C^{2,\alpha}(\ol{\Om}))$ with $k\geq 2$ and let $w_1,w_2$ have the same Cauchy data and solve $\Delta w_i + a_i(x,w_i) = 0$ in $\Omega$.
Write $q_i = \p_u a_i(x,w_i)$. Let $S_{a_1}\colon V_{q_1,\delta_1}\to C^{2,\alpha}(\ol{\Om})$ be the solution map from Lemma \ref{lemma_perturbed_sol}, for some $\delta_1>0$.
Suppose $u_{1,v} = S_{a_1}(v)$ and that $C_{a_1}^{w_1,\delta} \subseteq C_{a_2}^{0,C}$.
Then there exists a $\delta_2 > 0$ and a $C^k$ map $T_{a_2}\colon V_{q_1,\delta_2}\to C^{2,\alpha}(\ol{\Om})$, $T_{a_2}(v) = u_{2,v}$, where $u_{2,v}$ has the same Cauchy data as $u_{1,v}$ and solves $\Delta u_{2,v} + a_2(x,u_{2,v}) = 0$.
Moreover, when \( \partial_{u}a_{1}(x,w_{1}) = \partial_{u}a_{2}(x,w_{2}) \) then $T_{a_2}'(0)v = v$.
\end{Lemma}
\begin{proof}
First we use $C_{a_1}^{w_1,\delta} \subseteq C_{a_2}^{0,C}$ to find, for any $v \in V_{q_1,\delta_1}$, a function $u_{2,v}$ with the same Cauchy data as $u_{1,v}$ and solving $\Delta u_{2,v} + a_2(x,u_{2,v}) = 0$.
Note that $u_{1,0} = w_1$.
Moreover, both $u_{2,0}$ and $w_2$ solve the equation $\Delta u + a_2(x,u) = 0$ and they have the same Cauchy data, so by Lemma \ref{lem_carleman} one has $u_{2,0} = w_2$. By \eqref{lemma_estimate_r} we have 
\begin{equation} \label{uvw_diff_first}
\norm{u_{1,v}-w_1}_{C^{2,\alpha}(\ol{\Om})} \leq C \norm{v}_{C^{2,\alpha}(\ol{\Om})},
\end{equation}
Using this, Lemma \ref{lem_carleman}, $(u_{2,v}-w_2)|_{\p\Om}=(u_{1,v}-w_1)|_{\p\Om}$, $\p_{\nu}(u_{2,v}-w_2)|_{\p\Om}=\p_{\nu}(u_{1,v}-w_1)|_{\p\Om}$ and the fact that $\norm{u_{2,v}}_{C^{2,\alpha}(\ol{\Om})} \leq C$, we have 
\begin{equation} \label{uvw_diff_second}
\norm{u_{2,v}-w_2}_{C^{2,\alpha}(\ol{\Om})} \leq C \norm{v}_{C^{2,\alpha}(\ol{\Om})}.
\end{equation}

Let $r_v = u_{1,v}-u_{2,v}$.
Then $r_v$ satisfies
\begin{equation*}
	(\Delta+q_2) r_{v} = q_2r_{v} + a_2(x,u_{2,v}) - a_1(x,u_{1,v}) = q_2r_{v} + a_2(x,u_{1,v} - r_{v}) - a_1(x,u_{1,v}).
\end{equation*}
Let $G$ be the inverse of $\Delta +q_2: Y \to Z$ provided by Lemma \ref{lemma:schrödinger_bounded_inverse}.
Then $r_v$ solves the fixed point equation 
\begin{equation} \label{rv_eq}
	r_v = G(q_2r_{v} + a_2(x,u_{1,v} - r_{v}) - a_1(x,u_{1,v})).
\end{equation}

We would like to show that $r_v$ depends smoothly on $v$ by applying the implicit function theorem to \eqref{rv_eq}.
However, for a general function $r$ the expression $q_2r + a_2(x,u_{1,v} - r) - a_1(x,u_{1,v})$ might not be in the domain of $G$.
For this reason we introduce the projection $P\colon C^{2,\alpha}(\ol{\Om})\to Z\cap C^{2,\alpha}(\ol{\Om})$ from Lemma \ref{lemma:projection}.
Now define the map $F\colon V_{q_1,\delta_1}\times C^{2,\alpha}(\ol{\Om})\to C^{2,\alpha}(\ol{\Om})$ by
\begin{equation*}
	F(v,r) = r - GP(q_2r + a_2(x,u_{1,v} - r) - a_1(x,u_{1,v})).
\end{equation*}
Next we compute $F(0,w_1-w_2)$ and $D_rF(0,w_1-w_2;h)$ and find
\begin{align*}
	F(0,w_1-w_2) &= w_1-w_2 - GP(q_2(w_1-w_2)+a_2(x,w_1-(w_1-w_2)) - a_1(x,w_1)) \\
	&= w_1-w_2- GP(q_2(w_1-w_2) + \Delta(-w_2 + w_1)) \\
	&= w_1-w_2- GP((\Delta+q_2)(w_1-w_2)) \\
	&= w_1-w_2- G((\Delta+q_2)(w_1 - w_2)) = 0
\end{align*}
and
\begin{equation*}
	D_rF(0,w_1-w_2;h) = h - GP(q_2h - \p_ua_2(x,w_2)h) = h.
\end{equation*}
Since $h\mapsto D_rF(0,w_1-w_2;h)$ is bijective, it follows from the implicit function theorem \cite[Theorem 4]{hg27} that there exists a $\delta_2$ with $0<\delta_2 \leq \delta_1$ and a $C^k$ map $R\colon V_{q_1,\delta_2}\to C^{2,\alpha}(\ol{\Om})$ such that $\tilde{r} = R(v)$ is the unique solution to
\begin{equation}\label{eq:implicit_function_theorem_solution}
	\tilde{r} = GP(q_2\tilde{r} + a_2(x,u_{1,v} - \tilde{r}) - a_1(x,u_{1,v})).
\end{equation}
for $\tilde{r}$ close to $w_1-w_2$. Choosing $v\in V_{q_1,\delta_2}$ in $u_{1,v} = S_{a_1}(v)$, we find that $r_v$ is in the range of $R$ and that $r_v$ satisfies \eqref{eq:implicit_function_theorem_solution}. Moreover, by \eqref{uvw_diff_first} and \eqref{uvw_diff_second} we also have 
\[
\norm{r_v - (w_1-w_2)}_{C^{2,\alpha}(\ol{\Om})} \leq C \norm{v}_{C^{2,\alpha}(\ol{\Om})}.
\]
By the uniqueness of $\tilde{r} = R(v)$ near $w_1-w_2$ we have $r_v = R(v)$ for $v\in B_{\delta_2}$. Thus the map $v\mapsto r_v$ is indeed $C^k$.

Since $r_v = u_{1,v} - u_{2,v}$ we can define the $C^k$ map
\begin{equation*}
	T_{a_2}(v) \coloneqq S_{a_1}(v) - R(v).
\end{equation*}
It remains to show that $T_{a_2}'(0)v = v$, provided \( \partial_{u}a_{1}(x,w_{1}) = \partial_{u}a_{2}(x,w_{2}) \).
To do this, we use the implicit function theorem to compute $R'(0)$,
\begin{equation*}
	R'(0)v = -[D_rF(0,R(0))]^{-1}D_vF(0,R(0))v.
\end{equation*}
Since $D_rF(0,R(0))v = v$ and $D_vF(0,R(0))v = 0$ it follows that $R'(0)v = 0$.
Now we have
\begin{equation*}
	T_{a_2}'(0)v = S_{a_1}'(0)v + R'(0)v = S_{a_1}'(0)v = v. \qedhere
\end{equation*}
\end{proof}

\begin{Remark}
    The proof of Lemma \ref{lemma:smooth_solution_operator} is the only place where we need the assumption $\p_ua_i(x,w_i)\in C^{2,\alpha}(\ol{\Om})$. We need it to apply the projection $P$ to $(\Delta + q_2)(w_1-w_2)$.
\end{Remark}

\section{First linearization} \label{sec_first_linearization}

Throughout this section, we let $a_1, a_2 \in C^{3,\alpha}(\mR, C^{2,\alpha}(\ol{\Om}))$ and let $w\in C^{2,\alpha}(\ol{\Om})$ be a fixed solution of $\Delta w + a_1(x,w) = 0$ in $\Om$. Write $q = \p_u a_1(x,w)$ and consider the sets 
\begin{align*}
V_q &= \{ v \in C^{2,\alpha}(\ol{\Om}) \,:\, \Delta v + qv = 0 \text{ in $\Om$} \}, \\
V_{q,\delta} &= \{ v \in V_q \,:\, \norm{v}_{C^{2,\alpha}(\ol{\Om})} < \delta \}.
\end{align*}
Assume $C_{a_1}^{w,\delta} \subseteq C_{a_2}^{0,C}$. For any $v \in V_{q,\delta}$ with $\delta$ small, we let $u_{1,v} = S_{a_1,w}(v)$ and $u_{2,v} = T_{a_2,w}(v)$ be the solutions of $\Delta u_{j,v} + a_j(x,u_{j,v}) = 0$ given by Lemmas \ref{lemma_perturbed_sol} and \ref{lemma:smooth_solution_operator}.

\begin{Lemma}\label{lemma_first_lin}
Suppose that $C_{a_1}^{w,\delta} \subseteq C_{a_2}^{0,C}$. There is $\delta_1 > 0$ such that for any $v \in V_{q,\delta_1}$ one has 
\[
\p_u a_1(x, u_{1,v}(x)) = \p_u a_2(x, u_{2,v}(x)), \qquad x \in \ol{\Om}.
\]
\end{Lemma}
\begin{proof}
Let $v \in V_{q,\delta}$ and let $v_t = v + th$ where $h$ solves $\Delta h + qh = 0$ and $t$ is small. Consider the solutions $u_{1,v_t} = S_{a_1,w}(v_t)$ and $u_{2,v_t} = T_{a_2,w}(v_t)$ of 
\[
\Delta u_{j,v_t} + a_j(x,u_{j,v_t}) = 0.
\]
The solutions $u_{j,v_t}$ are $C^2$ with respect to $t$ and have the same Cauchy data. Differentiating the above equation in $t$ and writing $\dot{u}_j = \p_t u_{j,v_t}|_{t=0}$, we obtain 
\[
\Delta \dot{u}_j + \p_u a_j(x, u_{j,v}) \dot{u}_j = 0.
\]
Subtracting the equations for $j=1, 2$ and rewriting yields 
\begin{equation} \label{udot_diff_eq}
(\Delta + \p_u a_2(x,u_{2,v}))(\dot{u}_1-\dot{u}_2) + (\p_u a_1(x,u_{1,v}) - \p_u a_2(x,u_{2,v})) \dot{u}_1 = 0.
\end{equation}

Suppose that $\tilde{v}_2$ solves $(\Delta + \p_u a_2(x,u_{2,v})) \tilde{v}_2 = 0$. Integrating \eqref{udot_diff_eq} against $\tilde{v}_2$ and using that $\dot{u}_1-\dot{u}_2$ has zero Cauchy data gives 
\[
\int_{\Om} (\p_u a_1(x,u_{1,v}) - \p_u a_2(x,u_{2,v})) \dot{u}_1 \tilde{v}_2 \,dx = 0.
\]
It remains to study $\dot{u}_1 = D S_{a_1,w}(v)h$. By Lemma \ref{lemma_s_derivative}, when $v \in V_q$ is sufficiently small any solution $\tilde{v}_1$ of $(\Delta + \p_u a_1(x,u_{1,v})) \tilde{v}_1 = 0$ can be written as $D S_{a_1,w}(v)h$ for a suitable $h$. It follows that 
\[
\int_{\Om} (\p_u a_1(x,u_{1,v}) - \p_u a_2(x,u_{2,v})) \tilde{v}_1 \tilde{v}_2 \,dx = 0
\]
for any solutions $\tilde{v}_j$ of $(\Delta + \p_u a_j(x,u_{j,v})) \tilde{v}_j = 0$. Now it follows from the density of products of solutions as in the standard Calderón problem (see \cite{SylvesterUhlmann} for $n \geq 3$ and \cites{Bukhgeim, btw20} for $n=2$) that \( \partial_{u}a_{1}(x,u_{1,v}) = \partial_{u}a_{2}(x,u_{2,v}) \).
\end{proof}

\begin{Lemma}\label{lemma_first_lin_varphi}
In the setting of Lemma \ref{lemma_first_lin}, the function 
\[
\varphi_v = u_{2,v} - u_{1,v}
\]
is independent of $v \in V_{q,\delta_1}$.
\end{Lemma}

\begin{proof}
Write $\psi_t = \varphi_{tv}$. The function $\psi_t$ is $C^2$ in $t$, has zero Cauchy data on $\p \Om$, and satisfies 
\[
\Delta \psi_t = a_1(x,u_{1,tv}) - a_2(x,u_{2,tv}).
\]
Thus the derivative $z_t = \p_t \psi_t$ satisfies  
\[
\Delta z_t = \p_u a_1(x,u_{1,tv}) \p_t u_{1,tv} - \p_u a_2(x,u_{2,tv}) \p_t u_{2,tv}.
\]
Combining this with Lemma \ref{lemma_first_lin} yields 
\[
\Delta z_t = -\p_u a_1(x,u_{1,tv}) z_t.
\]
Since $z_t$ has zero Cauchy data, it follows that $z_t = 0$ and consequently $\psi_t$ is independent of $t$. In particular, $\varphi_v = \varphi_0$.
\end{proof}

We can now give the proofs of Theorem \ref{thm_main0} and \ref{thm_main01}.

\begin{proof}[Proof of Theorem \ref{thm_main01}]
Let $w_1$ solve $\Delta w_1 + a_1(x,w_1) =0$ and assume that $C_{a_1}^{w_1,\delta} \subseteq C_{a_2}^{0,C}$.  Using Lemma \ref{lemma_first_lin_varphi}, we have 
\begin{align*}
\Delta \varphi &= \Delta u_{2,v} - \Delta u_{1,v} = a_1(x,u_{1,v})-a_2(x,u_{2,v}) \\
 &= a_1(x,u_{1,v})-a_2(x,u_{1,v}+\varphi).
\end{align*}
This can be rewritten as 
\[
a_1(x,u_{1,v}(x)) = T_{\varphi} a_2(x,u_{1,v}(x)).
\]
It is enough to show that there is $\eps > 0$ such that for any $\bar{x} \in \ol{\Om}$ and $\lambda \in [-\eps,\eps]$, one can find a small solution $v$ such that 
\begin{equation} \label{concl1}
u_{1,v}(\bar{x}) = w_1(\bar{x}) + \lambda.
\end{equation}

Fix $x_0 \in \ol{\Om}$, and use Runge approximation (Lemma \ref{lemma:nonzero_solution}) to generate a solution $v = v_{x_0}$ of $\Delta v + \p_u a_1(x,w_1)v = 0$ with $v(x_0) = 4$. Let $U_{x_0}$ be a neighborhood of $x_0$ so that $v(x) \geq 2$ for $x \in \ol{U}_{x_0} \cap \ol{\Om}$. In the notation of Lemma \ref{lemma_perturbed_sol} one has 
\[
u_{1,tv} = w_1 + tv + Q_{a_1,w_1}(tv)
\]
where 
\[
\norm{Q_{a_1,w_1}(tv)} \leq C_{a_1,w_1} t^2 \norm{v}_{C^{2,\alpha}(\ol{\Om})}^2.
\]
Thus for $x \in \ol{U}_{x_0} \cap \ol{\Om}$ one has 
\[
|u_{1,tv}(x) - w_1(x)| \geq 2|t| - C_{a_1,w_1} t^2 \norm{v}_{C^{2,\alpha}(\ol{\Om})}^2.
\]
Set $\eps_{x_0} = 1/(C_{a_1,w_1} \norm{v}_{C^{2,\alpha}(\ol{\Om})}^2)$. Then for $|t| \leq \eps_{x_0}$ 
\[
|u_{1,tv}(x) - w_1(x)| \geq |t|.
\]
The next step is to use compactness to find a finite cover $\{ U_{x_1}, \ldots, U_{x_N} \}$ of $\ol{\Om}$ and to set 
\[
\eps = \min\{ \eps_{x_1}, \ldots, \eps_{x_N}, \delta_0 \}.
\]
Here $\delta_0$ is chosen so that $\norm{t v_{x_j}}_{C^{2,\alpha}} \leq \delta$ whenever $|t| \leq \delta_0$ and $1 \leq j \leq N$.

Now fix any $\bar{x} \in \ol{\Om}$ and $\lambda \in [-\eps,\eps]$, and choose $j$ so that $\bar{x} \in U_j$. Define 
\[
\eta(t) = u_{1,tv_{x_j}}(\bar{x}) - w_1(\bar{x}).
\]
Then $\eta: [-\eps, \eps] \to \mR$ is continuous with $\eta(\eps) \geq \eps$ and $\eta(-\eps) \leq -\eps$. By continuity, there is $\bar{t} \in [-\eps,\eps]$ such that $\eta(\bar{t}) = \lambda$. This proves that one has \eqref{concl1} for some choice of $v$, which proves the theorem.
\end{proof}

\begin{proof}[Proof of Theorem \ref{thm_main0}]
Since $w$ is a common solution for nonlinearities $a_1$ and $a_2$, we have $w_1=w_2=w$ in Theorem \ref{thm_main01}. Consequently $\varphi=0$ and $T_{\varphi} a_2 = a_2$. The result now follows from Theorem \ref{thm_main01}.
\end{proof}

\section{Higher order linearization}\label{sec_proofs_main}

In this section we prove theorem \ref{thm_main02}.
We use the higher order linearization method with the smooth solution maps from Sections \ref{sec_solvability} and \ref{sec_second_sol_map}.
Essentially the method is to show that the derivatives of order \( k \) of the solution maps satisfy a certain partial differential equation and have the same Cauchy data, provided \( \partial_{u}^{l}a_{1}(x,w_{1})=\partial_{u}^{l}a_{2}(x,w_{2}) \) for \( l\leq k-1 \).
Then Theorem \ref{thm_main02} follows from an integration by parts argument.
We start by proving that the derivatives of order \( k \) of the solution maps satisfy a certain differential equation in the following Lemma.

\begin{Lemma}\label{lemma:inductive_higher_linearization}
Let \( a_{1},a_{2} \in C^{k+2}(\mathbb{R},C^{2,\alpha}(\ol{\Omega}))\) with $k\geq 1$.
Let \( S_{a_{1}},T_{a_2} \) be the solution operators \( \Delta u + a_{i}(x,u)=0 \) from Lemma \ref{lemma_perturbed_sol} and Lemma \ref{lemma:smooth_solution_operator}.
Denote by \( R(v)\coloneqq S_{a_{1}}(v)-T_{a_{2}}(v) \) and let \( u = D^{l}R(0; v_{1},\ldots, v_{k}) \) for \( 1 \leq l \leq k \).
Then \( u \) satisfies
\begin{equation}\label{eq-inductive_higher_linearization_new_1}
\begin{cases}
	[\Delta + \partial_{u}a_{2}(x,w_{2})]u = F_{l} + [\partial_{u}^{l}a_{2}(x,w_{2})-\partial_{u}^{l}a_{1}(x,w_{1})]\prod_{i=1}^{l}h_{i}&\text{in }\Omega, \\
	u = \partial_{\nu}u = 0&\text{on }\partial\Omega.
\end{cases}
\end{equation}
where \( h_{i} = S_{a_{1}}^{\prime}(0;v_{i}) \), \( w_{1} = S_{a_{1}}(0) \), \( w_{2} = T_{a_{2}}(0) \), and
\begin{equation*}
	F_{l} = \sum_{i=2}^{l}\partial_{u}^{i}a_{2}(x,w_{2})P_{i}+\sum_{i=1}^{l-	1}[\partial_{u}^{i}a_{2}(x,w_{2})-\partial_{u}^{i}a_{1}(x,w_{1})]Q_{i}
\end{equation*}
whenever $2\leq l\leq k$, and otherwise $F_1 = 0$.
Here \( P_{i} \) is a polynomial in \( D^{j}S_{a_{1}}(0;v_{1},\ldots, v_{j}) \), \( D^{j}T_{a_{1}}(0;v_{1},\ldots, v_{j}) \), and \( D^{j}R(0;v_{1},\ldots, v_{j}) \) and \( Q_{i} \) is a polynomial in \( D^{j}S_{a_{1}}(0;v_{1},\ldots, v_{j}) \) for \( 1\leq j\leq i \).
Both \( P_{i} \) and \( Q_{i} \) vanish at the origin and each term in \( P_{i} \) contains at least one factor which is \( D^{j}R(0;v_{1},\ldots,v_{j}) \).
Additionally, \( F_{l} \) satisfies
\begin{equation}\label{eq-inductive_higher_linearization_new_2}
	\partial_{u}^{i}a_{1}(x,w_{1}) = \partial_{u}^{i}a_{2}(x,w_{2})\text{ for }1\leq i\leq l-1\implies F_{l} = 0.
\end{equation}
\end{Lemma}
\begin{proof}
Since \( S_{a_{1}}(v),T_{a_{2}}(v) \) have equal Cauchy data, \( R(v) \) has vanishing Cauchy data for every \( v \).
Then the derivatives of \( R(v) \) also have vanishing Cauchy data, which establishes the boundary condition in \eqref{eq-inductive_higher_linearization_new_1}.

Deriving the linearized differential equation is somewhat more involved.
The first linearization is
\begin{equation}\label{eq-inductive_higher_linearization_new_3}
	[\Delta+\partial_{u}a_{2}(x,w_{2})]u = [\partial_{u}a_{2}(x,w_{2})-\partial_{u}a_{1}(x,w_{1})]h_{1}.
\end{equation}
In other words, \( F_{1} = 0 \).
The second linearization is
\begin{equation*}
\begin{aligned}
	[\Delta+\partial_{u}a_{2}(x,w_{2})]u = &-\partial_{u}^{2}a_{2}(x,w_{2})[S_{a_{1}}^{\prime}(0;v_{1})R^{\prime}(0;v_{2}) + T_{a_{1}}^{\prime}(0;v_{2})R^{\prime}(0;v_{1})] \\
	&+[\partial_{u}a_{2}(x,w_{2})-\partial_{u}a_{1}(x,w_{1})]S_{a_{1}}^{\prime\prime}(0;v_{1},v_{2}) \\
	&+ [\partial_{u}^{2}a_{2}(x,w_{2})-\partial_{u}^{2}a_{1}(x,w_{1})]h_{1}h_{2}.
\end{aligned}
\end{equation*}
Recall that \( h_{i} = S_{a_{1}}^{\prime}(0,v_{i}) \).
We see that
\begin{equation*}
\begin{aligned}
	F_{2} = &-\partial_{u}^{2}a_{2}(x,w_{2})[S_{a_{1}}^{\prime}(0;v_{1})R^{\prime}(0;v_{2}) + T_{a_{1}}^{\prime}(0;v_{2})R^{\prime}(0;v_{1})] \\
	&+[\partial_{u}a_{2}(x,w_{2})-\partial_{u}a_{1}(x,w_{1})]S_{a_{1}}^{\prime\prime}(0;v_{1},v_{2}).
\end{aligned}
\end{equation*}
We proceed by induction.
Suppose that \( \tilde{u} = D^{l-1}R(0;v_{1},\ldots,v_{l-1}) \) satisfies
\begin{equation*}
	[\Delta + \partial_{u}a_{2}(x,w_{2})]\tilde{u} = F_{l-1} + [\partial_{u}^{l-1}a_{2}(x,w_{2})-\partial_{u}^{l-1}a_{1}(x,w_{1})]\prod_{i=1}^{l-1}h_{i},
\end{equation*}
and denote again by \( u \) the derivative \( D^{l}R(0,v_{1},\ldots,v_{l}) \).
We linearize this term by term.
The derivative of \( (\Delta+\partial_{u}a_{2}(x,w_{2}))\tilde{u} \) at \( 0 \) and in direction \( v_{l} \) is
\begin{equation*}
	(\Delta+\partial_{u}a_{2}(x,w_{2}))u + \partial_{u}^{2}a_{2}(x,w_{2})T_{a_{2}}^{\prime}(0;v_{l})\tilde{u},
\end{equation*}
and the latter term will be incorporated into \( F_{l} \).
In the following computations \( D_{v} \) denotes the derivative at \( 0 \) in direction \( v \).
The derivative of \( F_{l-1} \) at \( 0 \) and in direction \( v_{l} \) is
\begin{equation*}
\begin{aligned}
	&\sum_{i=2}^{l-1}[\partial_{u}^{i+1}a_{2}(x,w_{2})T_{a_{2}}^{\prime}(0;v_{l})P_{i} + \partial_{u}^{i}a_{2}(x,w_{2})D_{v_{l}}P_{i}] \\
	&+\sum_{i=1}^{l-2}[\partial_{u}^{i+1}a_{2}(x,w_{2})\underbrace{T_{a_{2}}^{\prime}(0;v_{l})}_{=S_{a_{1}}^{\prime}(0;v_{l})-R^{\prime}(0;v_{l})}-\partial_{u}^{i+1}a_{1}(x,w_{1})S_{a_{1}}^{\prime}(0;v_{l})]Q_{i} \\
	&+\sum_{i=1}^{l-2}[\partial_{u}^{i}a_{2}(x,w_{2})-\partial_{u}^{i}a_{1}(x,w_{1})]D_{v_{l}}Q_{i} \\
	&=\sum_{i=2}^{l}\partial_{u}^{i}a_{2}(x,w_{2})[(1-\delta_{i,2})(T_{a_{2}}^{\prime}(0;v_{l})P_{i-1} - R^{\prime}(0;v_{l})Q_{i-1}) + (1-\delta_{i,l})D_{v_{l}}P_{i}] \\
	&+\sum_{i=1}^{l-1}[\partial_{u}^{i}a_{2}(x,w_{2})-\partial_{u}^{i}a_{1}(x,w_{1})][(1-\delta_{i,1})S_{a_{1}}^{\prime}(0;v_{l})Q_{i-1} + (1-\delta_{i,l-1})D_{v_{l}}Q_{i}],
\end{aligned}
\end{equation*}
where \( \delta_{i,j} \) denotes the Kronecker delta.
Lastly, we differentiate \( [\partial_{u}^{l-1}a_{2}(x,w_{2})-\partial_{u}^{l-1}a_{1}(x,w_{1})]\prod_{i=1}^{l-1}h_{i} \) at 0 and in direction \( v_{l} \).
We get
\begin{equation*}
\begin{aligned}
	&[\partial_{u}^{l}a_{2}(x,w_{2})\underbrace{T_{a_{2}}^{\prime}(0;v_{l})}_{=S_{a_{1}}^{\prime}(0;v_{l})-R^{\prime}(0;v_{l})}-\partial_{u}^{l}a_{1}(x,w_{1})S_{a_{1}}^{\prime}(0;v_{l})]\prod_{i=1}^{l-1}h_{i} \\
	&+ [\partial_{u}^{l-1}a_{2}(x,w_{2})-\partial_{u}^{l-1}a_{1}(x,w_{1})]D_{v_{l}}\Big(\prod_{i=1}^{l-1}h_{i}\Big) \\
	&=[\partial_{u}^{l-1}a_{2}(x,w_{2})-\partial_{u}^{l-1}a_{1}(x,w_{1})]D_{v_{l}}\Big(\prod_{i=1}^{l-1}h_{i}\Big) -\partial_{u}^{l}a_{2}(x,w_{2})R^{\prime}(0;v_{l})\prod_{i=1}^{l-1}h_{i}\\
	& +[\partial_{u}^{l}a_{2}(x,w_{2})-\partial_{u}^{l}a_{1}(x,w_{1})]\prod_{i=1}^{l}h_{i}.
\end{aligned}
\end{equation*}
Here, the terms on the second to last row will be incorporated into \( F_{l} \) and the last row is not incorporated into \( F_{l} \).
Now let
\begin{equation*}
	\tilde{P}_{i} = (1-\delta_{i,2})(T_{a_{2}}^{\prime}(0;v_{l})P_{i-1} - R^{\prime}(0;v_{l})Q_{i-1}) + (1-\delta_{i,l})D_{v_{l}}P_{i} - \delta_{i,2}T_{a_{2}}^{\prime}(0;v_{l})\tilde{u}-\delta_{il}R^{\prime}(0;v_{l})\prod_{i=1}^{l-1}h_{i}
\end{equation*}
and
\begin{equation*}
\begin{aligned}
	\tilde{Q}_{i} = [(1-\delta_{i,1})S_{a_{1}}^{\prime}(0;v_{l})Q_{i-1} + (1-\delta_{i,l-1})D_{v_{l}}Q_{i}] + (1-\delta_{i,l-1})D_{v_{l}}\Big(\prod_{i=1}^{l-1}h_{i}\Big).
\end{aligned}
\end{equation*}
By using \( \tilde{P}_{i} \), \( \tilde{Q}_{i} \) in place of \( P_{i} \), \( Q_{i} \) in the definition of \( F_{l} \), we find that the linearization of order \( l \) indeed is \eqref{eq-inductive_higher_linearization_new_1}.

A careful investigation of the above procedure reveals that every term in \( P_{i} \) indeed must contain at least one derivative \( D^{j}R \).
To see that \eqref{eq-inductive_higher_linearization_new_2} holds, first note that if \( \partial_{u}a_{1}(x,w_{1}) = \partial_{u}a_{2}(x,w_{2}) \) then the right-hand side of the first linearization \eqref{eq-inductive_higher_linearization_new_3} vanishes and it follows by unique continuation that \( R^{\prime}(0;v_{j}) = 0 \), since it has vanishing Cauchy data.
Then every term in \( F_{2} \) has a factor which vanishes, hence \( F_{2} = 0 \).
We obtain \eqref{eq-inductive_higher_linearization_new_2} by iterating this procedure to higher order linearizations.
\end{proof}

\begin{proof}[Proof of Theorem \ref{thm_main02}]
Let \( l\geq 2 \) be an arbitrary integer and suppose that \( \partial^{j}_{u}a_{1}(x,w_{1}) = \partial^{j}_{u}a_{2}(x,w_{2}) \) for \( 1\leq j \leq l-1 \).
Then we have from Lemma \ref{lemma:inductive_higher_linearization} that
\begin{equation}\label{eq:higher_linearization_pde}
	\begin{cases}
		\Delta f + qf = [\partial_{u}^{l}a_{2}(x,w_{2})-\partial_{u}^{l}a_{1}(x,w_{1})] \prod_{i=1}^{l}v_{i} &\text{in }\Omega \\
	f = 0 &\text{on }\partial\Omega \\
	\partial_{\nu}f = 0 &\text{on }\partial\Omega
\end{cases}
\end{equation}
where \( v_{j} \) solve \( \Delta v_{j} + \partial_{u}a_{1}(x,w_{1})v_{j} = 0 \) for \( j\in\{1,\ldots,l\} \).
Let \( v_{l+1} \) solve \( \Delta v_{l+1} + \partial_{u}a_{1}(x,w_{1})v_{l+1} = 0 \).
Multiplying the differential equation \eqref{eq:higher_linearization_pde} by \( v_{l+1} \) and integrating by parts twice gives
\begin{equation*}
	\int_{\Omega}[\partial_{u}^{l}a_{2}(x,w_{2})-\partial_{u}^{l}a_{1}(x,w_{1})] \prod_{i=1}^{l+1}v_{i} \,dx = 0. \qedhere
\end{equation*}    
\end{proof}

\appendix

\section{Runge approximation}\label{sec_runge}

In the proof of Theorem \ref{thm_main01} we need to find a solution of the linearized equation which is nonzero in some fixed but arbitrary point of the domain. A few ways to achieve this are described in \cite[Remark 2.2]{LLLS2}. For the sake of completeness, we give a proof based on Runge approximation that is valid in our situation following \cite{LassasLiimatainenSalo_poissonembedding}.

\begin{Lemma}\label{lemma:nonzero_solution}
Let $\Omega \subseteq \mR^n$ be a bounded open set and let $q \in C^{\alpha}(\ol{\Om})$. For any $x_0 \in \ol{\Om}$, there is $u \in C^{2,\alpha}(\ol{\Om})$ solving $(-\Delta+q)u=0$ in $\Om$ with $u(x_0) \neq 0$.
\end{Lemma}
\begin{proof}
Let $\Omega_2$ be a large ball with $\ol{\Omega} \subseteq \Omega_2$, and extend $q$ as a function in $C^{\alpha}_c(\Omega_2)$. We may choose $\Omega_2$ in such a way that $0$ is not a Dirichlet eigenvalue of $-\Delta+q$ in $\Omega_2$ (see e.g.\ \cite[Lemma 3.2]{Stefanov1990}). Now by \cite[Theorem 1 in Section 5.4]{bjs}, there is a small ball $\Omega_1$ centered at $x_0$ and a function $u_0 \in C^{2,\alpha}(\ol{\Om}_1)$ solving $(-\Delta+q)u_0 = 0$ in $\Omega_1$ with $u_0(x_0) = 1$. By Runge approximation (see Lemma \ref{lemma_runge} below), there is $u \in C^{2,\alpha}(\ol{\Om}_2)$ solving $(-\Delta+q)u=0$ in $\Om_2$ with $u(x_0)$ arbitrarily close to $u_0(x_0) = 1$. This concludes the proof.
\end{proof}

It remains to prove the Runge approximation result. Since the approximation is in the $C(\ol{\Om}_1)$ norm, we need a notion of suitable weak solutions with measure data in the duality argument. Let $\Omega \subseteq \mR^n$ be a bounded open set with smooth boundary, let $q \in L^{\infty}(\Omega)$, and assume that $0$ is not a Dirichlet eigenvalue of $-\Delta+q$ in $\Omega$. If $\mu$ is a bounded linear functional on $C(\ol{\Om})$ (i.e.\ $\mu$ is a measure), we say that $u \in L^1(\Omega)$ is a \emph{very weak solution} of
\begin{equation} \label{veryweak_def}
(-\Delta+q)u = \mu \text{ in $\Omega$}, \qquad u|_{\p \Omega} = 0,
\end{equation}
if 
\[
\int_{\Om} u(-\Delta+q)\varphi \,dx = \mu(\varphi)
\]
for any $\varphi \in C^2(\ol{\Om}) \cap H^1_0(\Om)$.

\begin{Proposition} \label{prop_veryweak}
For any $p < \frac{n}{n-1}$ there is $C > 0$ such that for any bounded linear functional $\mu$ on $C(\ol{\Om})$, there is a unique very weak solution $u \in W^{1,p}(\Om)$ of \eqref{veryweak_def} satisfying 
\[
\norm{u}_{W^{1,p}(\Om)} \leq C \norm{\mu},
\]
where $\norm{\mu} = \sup_{\norm{\varphi}_{C(\ol{\Omega})} = 1} \,\abs{\mu(\varphi)}$.
\end{Proposition}
\begin{proof}
If $q \geq 0$ this follows from \cite[Theorem 9.1]{stampacchia}. In general we may replace $q$ by $q+\gamma$ where $\gamma > 0$ is a large constant, and use the part of \cite[Theorem 9.1]{stampacchia} where $\lambda$ is away from the spectrum.
\end{proof}

We can now prove the Runge approximation result.

\begin{Lemma} \label{lemma_runge}
Let $\Omega_1$, $\Omega \subseteq \mR^n$ be bounded open sets so that $\ol{\Om}_1 \subseteq \Om$, $\Om \setminus \ol{\Om}_1$ is connected, and $\Om$ has smooth boundary. Suppose that $q \in C^{\alpha}_c(\Om)$ and that $0$ is not a Dirichlet eigenvalue of $-\Delta+q$ in $\Om$. Consider the sets 
\begin{align*}
S_1 &= \{ u \in C^{2,\alpha}(\ol{\Om}_1), \ (-\Delta+q)u =0 \text{ in $\Om_1$} \}, \\
S &= \{ u \in C^{2,\alpha}(\ol{\Om}), \ (-\Delta+q)u =0 \text{ in $\Om$} \}.
\end{align*}
For any $u \in S_1$ and any $\varepsilon > 0$, there is $v \in S$ with $\norm{u-v|_{\Om_1}}_{C(\ol{\Om}_1)} \leq \varepsilon$.
\end{Lemma}
\begin{proof}
By the Hahn-Banach theorem \cite[Corollary 3.15]{conway}, it is enough to show that any continuous linear functional on $C(\ol{\Om}_1)$ that vanishes on $S|_{\Om_1}$ must also vanish on $S_1$. Thus, let $\mu$ be a continuous linear functional on $C(\ol{\Om}_1)$ that satisfies 
\begin{equation} \label{mu_assumption1}
\mu(v|_{\ol{\Om}_1}) = 0 \text{ for all $v \in S$.}
\end{equation}
We consider the extension defined by 
\[
\bar{\mu}: C(\ol{\Om}) \to \mR, \ \ \bar{\mu}(u) = \mu(u|_{\ol{\Om}_1}).
\]
By the Riesz representation theorem, $\bar{\mu}$ is a measure in $\ol{\Om}$ with $\mathrm{supp}(\bar{\mu}) \subseteq \ol{\Om}_1$.

We use Proposition \ref{prop_veryweak} to find a very weak solution $w \in W^{1,p}(\Om)$ of the problem  
\begin{equation} \label{w_mubar_eq}
(-\Delta+q)w = \bar{\mu} \text{ in $\Om$}, \qquad w|_{\p \Om} = 0.
\end{equation}
We use the assumption \eqref{mu_assumption1} and the unique continuation principle to prove that 
\begin{equation} \label{w_zero}
\text{$w=0$ in $\Om \setminus \ol{\Om}_1$.}
\end{equation}
Assuming \eqref{w_zero}, the proof can be concluded as follows. Since $\mathrm{supp}(w) \subseteq \ol{\Om}_1$, there exist $w_j \in C^{\infty}_c(\Om_1)$ with $w_j \to w$ in $W^{1,p}(\Om)$. Given any $u \in S_1$, we let $\bar{u}$ be some function in $C^{2,\alpha}_c(\Om)$ with $\bar{u}|_{\ol{\Om}_1} = u$ and compute 
\begin{align*}
\mu(u) = \bar{\mu}(\bar{u}) = \int_{\Om} w(-\Delta+q)\bar{u} \,dx = \lim \int_{\Om} w_j(-\Delta+q)\bar{u} \,dx = \lim \int_{\Om_1} w_j(-\Delta+q)u \,dx = 0.
\end{align*}
Thus $\mu|_{S_1} = 0$ as required.

It remains to prove \eqref{w_zero}. We begin by studying the regularity of $w$ near $\p \Om$. Choose a ball $\Om_2$ with $\ol{\Om} \subseteq \Om_2$ so that $0$ is not a Dirichlet eigenvalue, and a very weak solution $\tilde{w}$ of 
\[
(-\Delta+q)\tilde{w} = \tilde{\mu} \text{ in $\Om_2$}, \qquad \tilde{w}|_{\p \Om_2} = 0,
\]
where $\tilde{\mu}$ is the extension of $\bar{\mu}$ by zero to $\ol{\Om}_2$.
Using the definition of very weak solutions and the facts that $q \in C^{\alpha}_c(\Om)$ and $\mathrm{supp}(\bar{\mu}) \subseteq \ol{\Om}_1$, we see that $\Delta \tilde{w} = 0$ near $\p \Om$ in the sense of distributions. Hence $\tilde{w}$ is $C^{\infty}$ near $\p \Om$. Let $g \in C^{2,\alpha}(\ol{\Om})$ be the solution of 
\[
(-\Delta+q)g = 0 \text{ in $\Om$}, \qquad g|_{\p \Om} = -\tilde{w}|_{\p \Om}.
\]
Then both $w$ and $\tilde{w}|_{\Omega} + g$ are very weak solutions of \eqref{w_mubar_eq}, and by uniqueness one has 
\[
w = \tilde{w}|_{\Omega} + g.
\]
It follows that $w$ is $C^{2,\alpha}$ near $\p \Om$. Moreover, since $w$ is a $W^{1,p}$ solution of $(-\Delta+q)w=0$ in $\Om \setminus \ol{\Om}_1$, it follows from \cite[see section 4. Concluding remarks]{hr} that $w$ is $W^{1,2}$ and consequently $C^{2,\alpha}$ in $\Om \setminus \ol{\Om}_1$.

We now let $v \in S$ and choose $\chi \in C^{\infty}_c(\Om)$ such that $\chi=1$ near $\ol{\Om}_1$ and $w$ is $C^{2,\alpha}$ in $\mathrm{supp}(1-\chi) \cap \ol{\Om}$. Then 
\[
\int_{\Om} w(-\Delta+q)v \,dx = \int_{\Om} w(-\Delta+q)(\chi v) \,dx + \int_{\Om} w(-\Delta+q)((1-\chi) v) \,dx.
\]
We use the definition of very weak solutions in the first term, and since $w$ is regular in $\mathrm{supp}(1-\chi)$ we may integrate by parts in the second term. This yields 
\[
\int_{\Om} w(-\Delta+q)v \,dx = \bar{\mu}(\chi v) + \int_{\p \Om} (\p_{\nu} w) v \,dS = \mu(v|_{\ol{\Om}_1}) + \int_{\p \Om} (\p_{\nu} w) v \,dS.
\]
Since $v \in S$, we have $\mu(v|_{\ol{\Om}_1}) = 0$ by the assumption \eqref{mu_assumption1}. Since we can vary the Dirichlet data of $v \in S$, it follows that $\p_{\nu} w|_{\p \Om} = 0$. Thus $w$ in particular satisfies 
\[
(-\Delta+q)w = 0 \text{ in $\Om \setminus \ol{\Om}_1$}, \qquad w|_{\p \Om} = \p_{\nu} w|_{\p \Om} = 0.
\]
Since $w$ is $C^{2,\alpha}$ in $\Om \setminus \ol{\Om}_1$ and this set is connected, the unique continuation principle yields \eqref{w_zero}. This finishes the proof.
\end{proof}

\printbibliography

\end{document}